\documentclass[11pt,a4paper]{amsart}
\usepackage{amsmath,amsfonts,amsthm,amsopn,color,amssymb,enumitem}
\usepackage{palatino}
\usepackage{graphicx}
\usepackage{caption}
\usepackage{subcaption}
\usepackage[colorlinks=true]{hyperref}
\hypersetup{urlcolor=blue, citecolor=red, linkcolor=blue}
\usepackage{changepage}
\usepackage{float}

\usepackage{esint}
\usepackage{pgfplots}
\usepackage{mathrsfs}
\usetikzlibrary{arrows}

\newtheorem{theorem}{Theorem}[section]
\newtheorem{lemma}[theorem]{Lemma}

\newtheorem{proposition}[theorem]{Proposition}
\newtheorem{remark}[theorem]{Remark}

\newcommand{\e}{\varepsilon}

\DeclareMathOperator{\dist}{dist}

\DeclareMathOperator{\supp}{supp}

\renewcommand{\t}{\theta}

\def\bbm[#1]{\mbox{\boldmath $#1$}}
\newcommand{\beq }{\begin{equation}}
\newcommand{\eeq }{\end{equation}}

\def\sideremark#1{\ifvmode\leavevmode\fi\vadjust{\vbox to0pt{\vss
 \hbox to 0pt{\hskip\hsize\hskip1em
 \vbox{\hsize3cm\tiny\raggedright\pretolerance10000
  \noindent #1\hfill}\hss}\vbox to8pt{\vfil}\vss}}}%

\begin{document}

\title[Vectorial $p$-Stokes system]{Symmetry and Monotonicity results for solutions of vectorial $p$-Stokes systems}

\author{Rafael López-Soriano, Luigi Montoro, Berardino Sciunzi}

\email[Rafael López-Soriano]{ralopezs@math.uc3m.es}
\email[Luigi Montoro]{montoro@mat.unical.it}%
\email[Berardino Sciunzi]{sciunzi@mat.unical.it}
\address[R. López-Soriano,]{Departamento de Matem\'aticas,
Universidad Carlos III de Madrid, Av. Universidad 30, 28911 Legan\'es, Madrid, Spain}
\address[L. Montoro, B. Sciunzi]{Dipartimento di Matematica e Informatica, Università della Calabria,
Ponte Pietro Bucci 31B, 87036 Arcavacata di Rende, Cosenza, Italy}


\keywords{p-Stokes system, p-Laplacian system, comparison principle, moving plane method}

\subjclass[2020]{35J47, 35J92, 76A05}


\maketitle

\begin{abstract}
In this paper we shall study qualitative properties of a $p$-Stokes type system, namely  $$ -{\boldsymbol \Delta}_p{\boldsymbol u}=-\operatorname{\bf div}(|D{\boldsymbol u}|^{p-2}D{\boldsymbol u}) = {\boldsymbol f}(x,{\boldsymbol u})\,\, \mbox{ in $\Omega$},$$
where ${\boldsymbol \Delta}_p$ is the $p$-Laplacian vectorial operator. More precisely, under suitable assumptions on the domain $\Omega$ and the function $\boldsymbol{ f}$, it is deduced that system solutions are symmetric and monotone. Our main results are derived from a vectorial version of the weak and strong comparison principles, which enable to proceed with the moving-planes technique for systems. As far as we know, these are the first qualitative kind results involving vectorial operators.
\end{abstract}
\section{Introduction}
\setcounter{equation}{0}
\noindent 
Let $\Omega$ be a bounded smooth domain in $\mathbb{R}^n$ with $n\geq 2$. Along this work, we shall consider a vector field ${\boldsymbol u}\,:\, \Omega \rightarrow \mathbb{R}^N$, $N\geq 1$,  namely ${\boldsymbol u}=(u^1,\ldots,u^N)$, that is a weak $C^{1}(\overline{\Omega})$ solution to the $p$-Laplace system 
\beq\label{system}\tag{$p$-S}
\begin{cases}
-{\boldsymbol \Delta}_p{\boldsymbol u}=   {\boldsymbol f}(x,{\boldsymbol u})& \mbox{in $\Omega$}\\
{\boldsymbol u}= 0  & \mbox{on  $\partial \Omega$}, 
\end{cases}
\eeq
where the vectorial operator $-{\boldsymbol \Delta}_p{\boldsymbol u}$ is defined, for  smooth functions, as 
\beq\label{pLaplace}
-{\boldsymbol \Delta}_p{\boldsymbol u}=-\operatorname{\bf div}(|D{\boldsymbol u}|^{p-2}D{\boldsymbol u})
\eeq
with $p>1$. The vectorial function $ {\boldsymbol f}:\Omega \times \mathbb{R}^N\to\mathbb{R}^N$ will fulfill a set of assumptions \eqref{hpf} described in Section~\ref{preli}.

Systems involving the $p$-Laplacian operator were firstly addressed by J.-L. Lions in contributions from the sixties \cite{Lions1,Lions2}. They were introduced as a generalization of the classical Navier-Stokes system. Actually, this kind of vectorial operators arises in the study of non-Newtonian fluids.

In the spacial case $N=n$, if one replaces $D{\boldsymbol u}$ by $\overline{D}{\boldsymbol u}$ in the operator \eqref{pLaplace}, where
$$
\overline{D}\boldsymbol {u}= D{\boldsymbol u} + D^T {\boldsymbol u},
$$
then \eqref{system} models the stationary behavior of some fluids without kinematic pressure. In that case, the symmetric $p$-Laplacian represents the stress tensor, $\boldsymbol{u}$ is the velocity of the fluid and $\boldsymbol{f}$ is the external body force. We refer to the reader to the monograph \cite{BiAmHas} for precise details, which includes a complete discussion of power-law models. The evolutionary analogue problem has been analyzed in \cite{BerDiRu,Mal1,Mal2}, which studies the motion of the fluid. Let us emphasize  that the shear thinning fluids are the best modeled by a constitutive law, which corresponds to a shear exponent $p\in (1, 2]$.

Another operator typically considered in literature is the so-called modified $p$-Laplacian, namely $\operatorname{\bf div}(|D{\boldsymbol u} + \mu |^{p-2}D{\boldsymbol u}) $ with $\mu>0$. In that case the operator is neither degenerate nor singular. For that problem the existence of solutions and high regularity results is quite known, see \cite{BeCr1,BeCr2,Cri} for more details.

\

In this work, we will consider the operator \eqref{pLaplace} arising in the phenomenon described above. We shall restrict ourselves to the case of positive solutions to \eqref{system}, namely ${\boldsymbol u} >0$ in~$\Omega$, meaning that $u^\ell>0$ in $\Omega$ for all $\ell=1,\ldots, N$. In fact, as a first step in our analysis, we derive a strong maximum principle for $p$-Stokes type systems under quite general hypotheses on $\boldsymbol{f}$. This result implies that each non-negative solution is indeed strictly positive.

The regularity of solutions to $p$-Laplacian systems has been carried out in \cite{KuuMin}. In particular, the authors prove that a solution of \eqref{system} is $C^{1,\alpha}_{loc}(\Omega)$ for every $p>1$. In \cite{KuuMin2}, the authors obtain pointwise estimates in terms of Riesz and Wolff potentials, which allows one to deduce sharp regularity results for solutions. We also refer the reader to \cite{Minsur} for a survey on the regularity for the $p$-Laplacian and other general operators.  Second-order estimates are provided by Cianchi and Maz'ya in \cite{Cma} for the Dirichlet and Neumann boundary problems. In particular, they obtain that $|D\boldsymbol{u}|^{p-2} D\boldsymbol{u} \in W^{1,2}_{loc}(\Omega)$ under very general hypotheses on $\boldsymbol{f}$ and $p$. Such a natural condition will be supposed in our main results.

\

The goal of this paper is to obtain some qualitative properties which allows us to describe the solutions of the $p$-Stokes system. More precisely, we will prove monotonicity and symmetry for \eqref{system} by assuming some conditions on $\Omega$ and $\boldsymbol{f}$. In order to do it, we will use the so-called moving-planes procedure, which, to our knowledge, has not yet been applied to vectorial case.

Typically, the method relies on a combination of a weak and a strong comparison principles. The scalar case, i.e. \eqref{system} with $N=1$, has been extensively developed, starting from the pioneer works \cite{Dam,DamPa} for $1<p<2$. Monotonicity and symmetry results for problems involving the $p$-Laplacian operator were extended in \cite{DamSci}, where the versions of the weak and the strong comparison principles were generalized taking into account the set of the critical points of the solution.

The global nature of the operator $\boldsymbol{\Delta}_p \boldsymbol{u}$ presents a delicate issue in order to apply the method in a standard way. In fact, as far as we know, it was unknown maximum principles for systems like \eqref{system}, and therefore neither comparison principles. In the case that the vectorial function $\boldsymbol{f}$ is strictly increasing with respect to $x$, we are able to circumvent these difficulties and present a series of maximum and comparison principles. This assumption on  $\boldsymbol{f}$ enables to separate strictly sub and supersolutions, see Theorem~\ref{thm:strong} for more information. We point out that the approached followed in the work in the case that $1<p\leq 2$.

The critical set $\mathcal{Z}_u$ will play a crucial role to prove the underlying comparison principles. More precisely, this set is formed by the points where the derivatives of the solution ${\boldsymbol u}$ vanish, i.e.,
\begin{equation}\label{eq:critraf}
\mathcal Z_{u}:=\{x\in \Omega \, :\, D {\boldsymbol u }=0\}.
\end{equation}
Along this work, we show that it is a zero-measure set and $\Omega\setminus \mathcal{Z}_u$ is connected, see Propositions~\ref{lemmaZuzero} and \ref{lem:connesso}. 

\

Now, we can state our main theorem:

\begin{theorem}\label{thm:main1}
Let $\Omega$ be a bounded smooth domain of $\mathbb R^n$, $n\geq 2$, convex with respect to the $x_n$-direction and symmetric with respect to $T_0$, where
\begin{equation}\label{eq:T0}
 T_0=\{x\in \Omega : x_n=0\}.
 \end{equation}
 Let ${\boldsymbol u}$  be a $C^1(\overline\Omega)$ weak solution to \eqref{system} with $1<p\leq2$ such that $$|D\boldsymbol{u}|^{p-2} D\boldsymbol{u} \in W_{loc}^{1,2}(\Omega),$$ ${\boldsymbol f}$ verifies \eqref{hpf},
\begin{equation}\label{eq:fcresc}
\boldsymbol f(x_1,\ldots,x_n,\boldsymbol u)<\boldsymbol f(x_1,\ldots,y_n,\boldsymbol u) \quad \mbox{ where } x_n<y_n<0, 
\end{equation}
and
\begin{equation}\label{eq:fsym}
\boldsymbol f(x_1,\ldots,x_{n-1},-x_n,\boldsymbol u)=\boldsymbol f(x_1,\ldots,x_{n-1},x_n,\boldsymbol u).
\end{equation}
 
\noindent Then $\boldsymbol u$ is symmetric with respect to~$T_0$ and non-decreasing with respect to the $x_n$-direction in~$\Omega_0=\{x\in \Omega \, :\, x_n<0\}$. Moreover,
\begin{equation}\label{eq:derivpositfollows}
{\partial_{x_n} \boldsymbol u}\geq 0\quad \text{in }\Omega_0.
\end{equation}
\end{theorem}

Let us also analyze the particular case of a $p$-Stokes system considered in a pipe-type domain $\mathcal{P}$. We define a pipe-shaped domain as
$$
\mathcal{P}=(-1,1)\times \mathcal{P}^{\bot },
$$
where the section of the tube $\mathcal{P}^\bot$ is a subdomain of $\mathbb{R}^{n-1}$. It will be also assumed that $\mathcal{P}^\bot$ satisfies
\begin{enumerate}[label=$({\mathcal{P}^*})$, ref=$({\mathcal{P}^*})$]
\item \label{hpP}

\

\begin{itemize}
\item[$(i)$]  $\mathcal{P}^\bot$ is convex with respect to $x_n$-direction; 
\item[$(ii)$]  $\mathcal{P}^\bot$ is symmetric with respect to the hyperplane $T_0$, defined in \eqref{eq:T0}.
\end{itemize}
\end{enumerate}

 We shall distinguish between the boundaries of $\mathcal{P}$ as
$$
\partial \mathcal{P} = \mathcal{P}^{\bot}_{-1} \cup \mathcal{P}^{\bot}_{1}\cup \partial_w  \mathcal{P},
$$
where the start and end edges are $ \mathcal{P}^{\bot}_{i}=\{i\}\times \mathcal{P}^{\bot}$ with $i=-1,1$ and the tube wall is $\partial_w  \mathcal{P}= (-1,1)\times \partial  \mathcal{P}^{\bot} $, see~Figure below. Next, let us present a second system defined in a pipe domain with Dirichlet boundary condition on the tube walls
\beq\label{system2}\tag{$p$-S2}
\begin{cases}
-{\boldsymbol \Delta}_p{\boldsymbol u}=   {\boldsymbol f}(x,{\boldsymbol u})& \mbox{in $\mathcal{P}$}\\
{\boldsymbol u}= 0  & \mbox{on  $\partial_w \mathcal{P}$}.
\end{cases}
\eeq

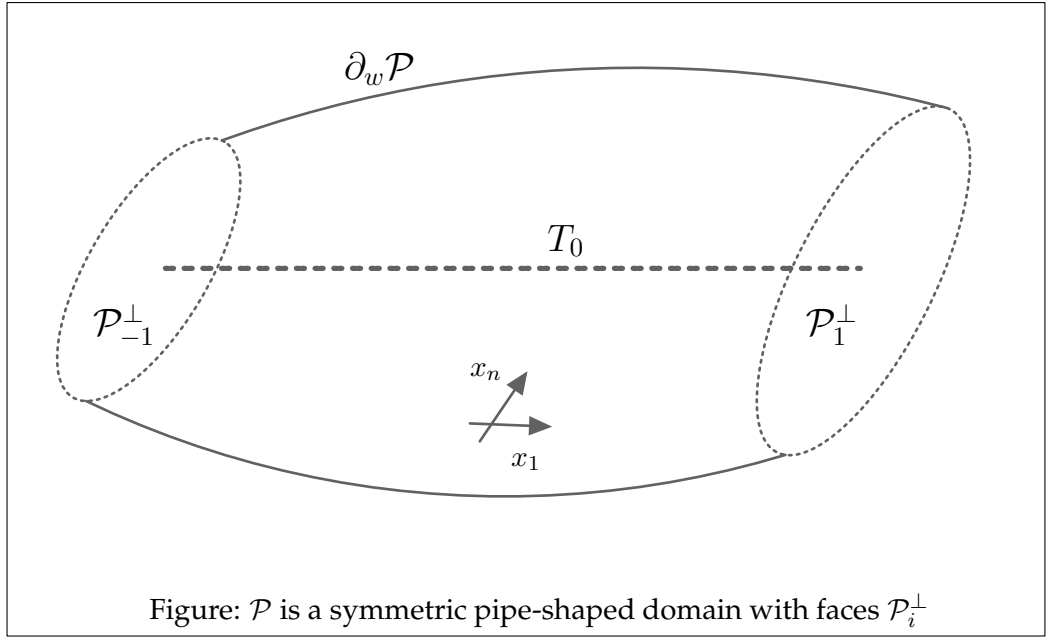
\definecolor{wrwrwr}{rgb}{0.3803921568627451,0.3803921568627451,0.3803921568627451}

\begin{figure}[H]
\begin{adjustwidth*}{-2cm}{-2cm} 
\begin{tikzpicture}[scale=0.7,line cap=round,line join=round,>=triangle 45,x=1cm,y=1cm]
\clip(-11.047297720595553,-9.449659233791984) rectangle (21.66460545417252,9.225189081482094); 
\draw [rotate around={59.58952885881112:(-4.837036161116002,1.4399184672730554)},line width=1pt,dotted,color=wrwrwr] (-4.837036161116002,1.4399184672730554) ellipse (2.8090178423254955cm and 1.1246459317152533cm);
\draw [rotate around={63.510636730696234:(8.59,1.23)},line width=1pt,dotted,color=wrwrwr] (8.59,1.23) ellipse (3.63268110036997cm and 1.3173351801971749cm);
\draw [shift={(4.3259630032569,-17.454222049042038)},line width=1pt,color=wrwrwr]  plot[domain=1.311768339868434:1.9208426327730799,variable=\t]({1*22.720066841030565*cos(\t r)+0*22.720066841030565*sin(\t r)},{0*22.720066841030565*cos(\t r)+1*22.720066841030565*sin(\t r)});
\draw (-5.3,1) node[below] {\Large{$\mathcal{P}^\bot_{-1}$}};
\draw (8,1) node[below] {\Large{$\mathcal{P}^\bot_{1}$}};
\draw (3,1.5) node[above] {\Large{$T_{0}$}};
\draw (-0.5,4.75) node[above] {\Large{$\partial_w\mathcal{P}$}};
\draw (2.25,-2.55) node[above] {{$x_1$}};
\draw (1.5,-0.85) node[above] {{$x_n$}};
\draw [shift={(1.855059584917851,15.187459470574343)},line width=1pt,color=wrwrwr]  plot[domain=4.260955368840012:5.0088547653271425,variable=\t]({1*18.043399359244017*cos(\t r)+0*18.043399359244017*sin(\t r)},{0*18.043399359244017*cos(\t r)+1*18.043399359244017*sin(\t r)});
\draw [line width=2pt,dash pattern=on 1pt off 1pt on 1pt off 4pt,color=wrwrwr] (-4.523468259465059,1.464614698810252)-- (8.555109285834464,1.464614698810252);
\draw [->,line width=1pt,color=wrwrwr] (1.3819887408947766,-1.8127593432742708) -- (2.278628809012238,-0.48325855261734185); 
\draw [->,line width=1pt,color=wrwrwr] (1.1964770026635771,-1.4726544898504053) -- (2.742408154590235,-1.5344917359274717);
\draw (2.5,-4.5)  node[below] {Figure: $\mathcal{P}$ is a symmetric pipe-shaped domain with faces $\mathcal{P}^\bot_i$};
\draw  (-7.5,-5.5) rectangle (12,6.5) ;
\end{tikzpicture}
\end{adjustwidth*}
\end{figure}

\vspace{-2cm}

The following result states that the monotonicity and symmetry for solutions of \eqref{system2} are in fact inherited from its behavior on the faces $\mathcal{P}_i^\bot$.

\begin{theorem}\label{thm:main2}
Let $\mathcal{P}$ be a bounded pipe-shaped domain of $\mathbb R^n$, $n\geq 2$ satisfying \ref{hpP}. Let ${\boldsymbol u}$  be a  $C^1(\overline\Omega)$ weak solution to \eqref{system2} with $1<p\leq2$ such that $|D\boldsymbol{u}|^{p-2} D\boldsymbol{u} \in W_{loc}^{1,2}(\Omega)$, ${\boldsymbol f}$ verifies \eqref{hpf}, \eqref{eq:fcresc} and \eqref{eq:fsym}. Suppose that
\begin{equation}\label{eq:monface}
{\partial_{x_n} \boldsymbol u}\geq 0  \, \, \mbox{ in } \mathcal{P}_i^\bot \cap \overline{\mathcal{P}_0} \, \mbox{ and } \, \boldsymbol{u}(i,x_2,\cdots,x_{n-1},x_n)=\boldsymbol{u}(i,x_2,\cdots,x_{n-1},-x_n)
\end{equation} 
for $i=-1,1$ and $\mathcal{P}_0=\{x\in \mathcal{P} \, :\, x_n<0\}$. 
\noindent Then $\boldsymbol u$ is symmetric with respect to~$T_0$ and non-decreasing with respect to the $x_n$-direction in~$\mathcal{P}_0$. Moreover,
\begin{equation*}
{\partial_{x_n} \boldsymbol u}\geq 0\quad \text{in}\quad\mathcal{P}_0.
\end{equation*}
\end{theorem}

\

The paper is organized as follows. In Section~\ref{preli} we set the notation, prove some preliminary results, including a Hopf's type lemma and strong maximum principle, which allows us to derive the underlying weak and strong comparison principles. The proofs of our main theorems are completed in Section~\ref{mainproof}.

\section{Preliminaries}\label{preli}
\setcounter{equation}{0}
\noindent 
\subsection{Notation} Generic fixed numerical constants will be denoted by $C$ (with subscript in some case) and will be allowed to vary within a single line or formula. Moreover $f^+$  will stand for the positive a part of a function, i.e.
$f^+=\max\{f,0\}$. We also denote $|A|$ the Lebesgue measure of the set $A$.

We will use the bold style to stress the vectorial nature of  different quantities. For instance, a $N$-vectorial function $\boldsymbol{w}$ defined in $\Omega$ will be written as
$$
\boldsymbol{w}(x)=\left(w^1(x),\ldots,w^N(x) \right),
$$
where $w^\ell$ is a scalar functions defined in $\Omega$ for $\ell=1,\ldots,N$.

\

We say that a vector field ${\boldsymbol u}$ solves weakly  \eqref{system} if and only if ${\boldsymbol u} \in W^{1,p}_{0}(\Omega)$ and 
\beq\label{weak}
\int_{\Omega} |D{\boldsymbol u}|^{p-2}D {\boldsymbol u} : D {\boldsymbol \varphi}\, dx = \int_{\Omega} {\boldsymbol f}(x,{\boldsymbol u}) \cdot \boldsymbol{\varphi}\, dx, \quad \forall  {\boldsymbol \varphi} \in W^{1,p}_{0}(\Omega).
\eeq

Here 
\[D {\boldsymbol u}=  \begin{pmatrix}
         \nabla u^1 \\
         \vdots \\
         \nabla u^N
        \end{pmatrix},\] 
        where 
\[\nabla u^\ell = \left(\frac{\partial u^\ell}{\partial x_1},\ldots,\frac{\partial u^\ell}{\partial x_n}\right)\,\, \text{for } \ell=1,\ldots,N,\]
and
\[\displaystyle{|D {\boldsymbol u}|=\sqrt{\sum_{\ell=1}^N\sum_{j=1}^n \left(\frac{\partial u^\ell}{\partial x_j}\right)^2}}.\] 
Sometimes it will be convenient to use the notation $\displaystyle u_j^\ell=\dfrac{\partial u^\ell}{\partial x_j}$.

We point out that along this paper the symbol $\,:\,$ stands for the scalar product of the matrices rows, namely 
$$
\mathcal{M}:\mathcal{N}=\sum_{i=1}^{q} \mathcal{M}^i \cdot \mathcal{N}^i=\sum_{i=1}^{q} \sum_{j=1}^{r} \mathcal{M}^i_j \mathcal{N}^i_j,
$$
where $\mathcal{M},\mathcal{N}$ are $q \times r$ real matrices, whereas $ \cdot$ denotes the scalar product of two real vectors. In the following we will denote the space of the $q\times r$ matrices as $\mathbb{R}^{q\times r}$ with $q,r \in \mathbb{N}$.
Observe that \eqref{weak} can be rewritten as
$$
\sum_{\ell=1}^N \int_{\Omega} |D {\boldsymbol u}|^{p-2}\nabla u^\ell  \cdot \nabla \varphi^\ell \, dx = \sum_{\ell=1}^N \int_{\Omega} f^\ell(x, {\boldsymbol u}) \varphi^\ell\, dx, \quad \forall {\bf\varphi} \in W^{1,p}_{0}(\Omega).
$$

In order to deduce qualitative properties of the solutions to \eqref{system}, we will exploit comparison principles. In this sense, we say that two vector fields ${\boldsymbol u },{\boldsymbol v } \in C^{1}(\overline \Omega)$ satisfy weakly the inequality
\beq\label{ineq}
-{\boldsymbol \Delta}_p {\boldsymbol u} - {\boldsymbol f}(x, {\boldsymbol u}) \leq -{\bf\Delta}_p {\boldsymbol v} - {\boldsymbol f}(x, {\boldsymbol v})\quad \mbox{in $\Omega$},
\eeq
if and only if
\begin{equation*}
\int_{\Omega} |D {\boldsymbol u}|^{p-2}D {\boldsymbol u}: D {\boldsymbol \varphi} \, dx- \int_{\Omega} {\boldsymbol f}(x,{\boldsymbol u})\cdot \boldsymbol{\varphi} \, dx \leq \int_{\Omega} |D {\boldsymbol v}|^{p-2}D {\boldsymbol v}: D {\boldsymbol \varphi} \, dx- \int_{\Omega} {\boldsymbol f}(x,{\boldsymbol v})\cdot \boldsymbol{\varphi}\, dx,
\end{equation*}
for every ${\boldsymbol \varphi} \in W^{1,p}_{0}(\Omega)$ such that ${\boldsymbol \varphi} \geq0$ a.e. in $\Omega$.

\

During this work the following hypotheses on ${\boldsymbol f}$ will be assumed:

\

\

\begin{enumerate}[label=$({hp^*})$, ref=${hp^*}$]
\item \label{hpf}

\

\begin{itemize}
\item[$(i)$] ${\boldsymbol f}(x,\cdot)$ is a locally Lipschitz continuous vector field, uniformly with respect to $x$, that is, each component $\ell\in\{1,\ldots,N\}$ \[f^\ell(x,t_1,\ldots, t_N):  \overline\Omega\times \mathbb{R}^N\rightarrow \mathbb R\] is a Lipschitz function with respect variable to $t_j$ namely, for every  $\Omega'\subseteq \overline\Omega$ and for every $M>0$,  there is a  positive constant $L_{\ell}=L_{\ell}(M,\Omega')$ such that for every $x \in \Omega'$ and every $ t_j,s_j \in [0,M]$ it holds:
$$  \vert f^\ell(x,t_1,\ldots,t_j,\ldots, t_N)- f^\ell(x,t_1,\ldots,s_j,\ldots, t_N)\vert \leq L_\ell \vert t_j-s_j \vert. $$

\item [$(ii)$] ${\boldsymbol f}$ is a positive vector field, namely
 \[f^\ell(x,t_1,\ldots,t_N)>0\] 
for all $x\in \Omega$ and for every $t_j>0$. 
\item[$(iii)$] ${\boldsymbol f} (x,\cdot)$ is not decreasing for a.e. $x\in \Omega$. More precisely, if $t_j\leq s_j$, then   
$$f^\ell(x,t_1,\ldots,t_j,\ldots, t_N)\leq f^\ell(x,t_1,\ldots,s_j,\ldots, t_N).$$\end{itemize}
\end{enumerate}

Now, let us state a useful  inequality  (see for example \cite{Dam}). For given $\eta,\eta' \in \mathbb{R}^{q \times r}$,  there exists a positive constant $C_p=C_p(p)$ such that
\begin{equation}\label{eq:lucio}
(|\eta|^{p-2}\eta-|\eta'|^{p-2}\eta'): (\eta - \eta') \geq C_p (|\eta|+|\eta'|)^{p-2}|\eta-\eta'|^2.
\end{equation}

The proof follows directly by identifying the matrix space $\mathbb{R}^{q\times r}$ with the vectorial space $\mathbb{R}^{q r}$.

\

\subsection{Comparison principles} 
Before proving the comparison principles, let us start by presenting the following version of Hopf's type lemma.
\begin{theorem}[Hopf's Lemma]\label{thm:hopf}
Assume that $\Omega$ is a domain in $\mathbb{R}^n$. Let $y \in \partial \Omega$ and ${\boldsymbol u}\in C^1(\Omega) \cap C^1(y)$ a non-negative weak solution to 
\begin{equation}\label{eq:st1}
-\operatorname{\bf div}(|D{\boldsymbol u}|^{p-2}D{\boldsymbol u}) +k {\boldsymbol u}^q= {\bf h}(x,{\boldsymbol u}) \quad \mbox{in $\Omega$},
\end{equation}
where ${\boldsymbol u}^q=\left((u^1)^q,\cdots,(u^N)^q\right)$, ${k\geq 0}$, $q\geq p-1$ and ${\bf h}:\Omega\times \mathbb{R}^N\to \mathbb{R}^N$ with $h^\ell(x,{\boldsymbol u})\geq 0$ in $\Omega \times \mathbb R^N$ for any $\ell\in \{1,\ldots,N\}$. Moreover, suppose that ${u^\ell}>0$ in $\Omega$ and that $u^\ell(y)=0$. If $\Omega$ satisfies an interior sphere condition at $y$, then 
\begin{equation}
\partial_\nu u^\ell(y) <0,
\end{equation}
where $\nu$ is the outward normal vector at $y$.
\end{theorem}
\begin{proof}
Using  the interior sphere condition there exists an open ball $B_r:=B_r(x_0)\subset \Omega$ with $y\in \partial \Omega\cap \partial B_r$. As customary in the case of a single equation, also in the nonlinear  vectorial case, we must define a suitable comparison function $\boldsymbol{v}$. Let us define the vector ${\boldsymbol v}:\Omega\subset\mathbb R^n\rightarrow \mathbb R^N$
\begin{equation}\nonumber
{\boldsymbol v}=(0,\ldots,v^\ell,\ldots,0),
\end{equation}
where $v^\ell=v^\ell(r)$, $r=|x-x_0|$ is the radial solution given in \cite{PSB}[Lemma 4.2.3]. In particular, in our context $v^\ell$ is a solution to
\begin{equation}\nonumber
\begin{cases}
-\operatorname{div}(|\nabla v^\ell|^{p-2}\nabla v^\ell)+ k_\ell (v^\ell)^q=0, & \text{in} \,\, A_r:=B_r\setminus \overline B_{r/2}\\
v^\ell=0\quad \text{on}\,\, \partial B_r \qquad v^\ell=m_\ell\quad \text{on}\,\, \partial B_{\frac r2},
\end{cases} 
\end{equation}
where $k_\ell,m_\ell$ are positive constants and $q\geq p-1$. Moreover by \cite{PSB}[Lemma 4.2.3] we have deduce that $|\nabla v^\ell|>0$ in $\overline A_r$ and in particular
\begin{equation}\nonumber
\partial_\nu v^\ell <0 \quad \text{on}\,\, \partial B_r,
\end{equation}
where $\nu$ is the exterior unit normal to $B_r$. By hypothesis we know that $u^\ell>0$ in $A_r\subset \Omega$. Let $m_\ell:=\inf \{u^\ell(x) \, : \, x\in B_{\frac r2} \}$ and $k_\ell\geq k$. 
Using \eqref{eq:st1} we infer that
\begin{equation}\label{eq:st3}
-\operatorname{\bf div}(|D{\boldsymbol v}|^{p-2}D{\boldsymbol v}) +k_\ell {\boldsymbol v}^q
\leq  -\operatorname{\bf div}(|D{\boldsymbol u}|^{p-2}D{\boldsymbol u}) +k {\boldsymbol u}^q, 
\end{equation}
in $A_r$, with ${\boldsymbol v}^q=(0,\cdots,(v^\ell)^q,\cdots,0)$ and
\begin{equation}\label{eq:st2}
v^\ell=0\quad \text{on}\,\, \partial B_r \qquad v^\ell=m_\ell\quad \text{on}\,\, \partial B_{\frac r2}. \end{equation}
Next, define the vectorial function
$$ 
(\boldsymbol{ v - u})^+=\left(0,\ldots,(v^\ell-u^\ell)^+,\ldots, 0 \right).
$$

 Using \eqref{eq:st2} we have that ${\boldsymbol \varphi}:=\chi_{A_r} (\boldsymbol{ v - u})^+\in W^{1,p}_{0}(A_r, \mathbb R^{N})$ is a suitable test function for~\eqref{eq:st3}. Integrating the previous expression, we obtain the inequality
\begin{equation*}
\int_{A_r} |D {\boldsymbol v}|^{p-2}D {\boldsymbol v}: D {\boldsymbol \varphi} \, dx + k_\ell \int_{A_r}{\boldsymbol v}^q\cdot \boldsymbol{\varphi} \, dx \leq \int_{A_r} |D {\boldsymbol u}|^{p-2}D {\boldsymbol u}:D {\boldsymbol \varphi} \, dx + k  \int_{A_r}{\boldsymbol u}^q\cdot \boldsymbol{\varphi} \, dx
\end{equation*}
and then
\begin{equation*}
\int_{A_r} (|D {\boldsymbol v}|^{p-2}D {\boldsymbol v}-|D {\boldsymbol u}|^{p-2}D {\boldsymbol u}: D (\boldsymbol{ v-u})^+) \, dx \leq   k  \int_{A_r} ({\boldsymbol u}^q-{\boldsymbol v}^q)({\boldsymbol v} -{\boldsymbol u})^+ \, dx \leq 0.
\end{equation*}

Using \eqref{eq:lucio} we deduce that
\begin{equation*}
\int_{A_r} (|D {\boldsymbol v}|+|D{\boldsymbol u}|)^{p-2}|D (\boldsymbol{ v-u})^+|^2 \, dx\leq 0.
\end{equation*}
Then ${\boldsymbol v}\leq {\boldsymbol u}$ in $A_r$. In particular, $u^\ell\geq v^\ell$ in $A_r$ and $u^\ell(y)-v^\ell(y)=0$. Hence $\partial_\nu(u^\ell-v^\ell)(y)\leq 0$, namely
$$\partial_\nu u^\ell(y)\leq \partial_\nu v^\ell(y)<0.$$
\end{proof}

Let us point out that previous lemma applies on each component regardless of the rest, according to the boundary conditions. Obviously, in case of the problem \eqref{system}, all the components have a boundary condition prescribed, so the result applies for  $\boldsymbol{u}$. As an immediate consequence of the Hopf's lemma, one can deduce a strong maximum principle.

\begin{theorem}[Strong maximum principle]
Suppose that $\Omega$ is a connected domain in $\mathbb{R}^n$. Let ${\boldsymbol u}$ be a $C^1(\Omega)$ non-negative weak solution to \eqref{eq:st1}. Then  for all $\ell\in \{1,\ldots,N\}$
\begin{equation*}
\mbox{either}\,\, u^\ell\equiv 0\quad  \mbox{in}\,\,\Omega\qquad\qquad  \mbox{or}\qquad\qquad u^\ell>0\quad\mbox{in }\Omega.
\end{equation*}
\end{theorem}

\begin{proof}
Consider the set $\mathcal{O}^\ell=\{x\in\Omega: u^\ell(x)=0\}$. Since $u^\ell$ is continuous, $\mathcal{O}^\ell$ is a closed set. Suppose that $u^\ell\not\equiv 0$ and  $\mathcal{O}^\ell\neq\emptyset$. Let $p\in \Omega\setminus \mathcal{O}^\ell$ such that $\dist(p,\mathcal{O}^\ell)<\dist(p,\partial \Omega)$. Define the largest ball $B_r(p)$ contained in $\mathcal{O}^\ell$. 
Therefore there exists a point $q\in \partial B_r(p) \cap \mathcal{O}^\ell $. Therefore, since $u^\ell(q)=0$, we are in position to apply the Theorem~\ref{thm:hopf}, in order to obtain that $\partial_\nu u^\ell (q)<0$. This is a contradiction with the fact that $q$ is a minimum of $u^\ell$ and $u$ is continuously differentiable at $y$.
\end{proof}

We are in position now to state the comparison principle, which are the basis of the proof of the main theorems. However, let us first introduce some notation. Denote a generic point  $x\in \mathbb R^n$ as $x=(x', x_n)$ with $x'=(x_1, \ldots, x_{n-1})$.
Moreover, let us define the following sets
\[I_{x'}:=\Big\{\Omega \cap \{(x',s), \,\, s\in \mathbb R\} \Big\} \]
and
\[\Omega':= \Big \{ x'\in \mathbb R^{n-1}\,:\, I_{x'}\neq \emptyset\Big\}.\]

Exploiting the fundamental theorem of calculus on the variable $x_n$, we prove a weak comparison principle for solutions of \eqref{system} in small neighborhoods of the $x_n$-foliations of the domain in case that $p\leq 2$.

\begin{theorem}[Weak Comparison Principle]\label{thm:weak1}
Let $\boldsymbol u, \boldsymbol v\in C^1(\overline\Omega)$  
be weak sub-supersolutions to
\begin{equation}\label{eq:(20)0}
-\boldsymbol \Delta_p \boldsymbol u \leq \boldsymbol f(x,\boldsymbol u) \quad\text{and}\quad -\boldsymbol \Delta_p\boldsymbol v\geq \boldsymbol f(x, \boldsymbol v) \quad \mbox{in $\tilde\Omega$},\end{equation} where $\tilde\Omega\subseteq \Omega$,  $\boldsymbol f$ satisfies  \eqref{hpf} and $1<p\leq2$.
Assume that $\boldsymbol u\leq \boldsymbol v$ on $\partial \tilde\Omega$ and for  $x'\in \Omega'$ and for all $\ell=1, \ldots, N$
\[\supp (u^\ell-v^\ell)^+ \cap I_{x'}\subseteq Z_{x'}^{\varepsilon} \cup L_{x'}^\theta,\] 
where  $Z_{x'}^{\varepsilon}, L_{x'}^\theta$ are subsets of $I_{x'}$ such that  $Z_{x'}^{\varepsilon} \cap L_{x'}^\theta=\emptyset$,
\[\bigcup_{x'\in \Omega'} Z_{x'}^{\varepsilon} \cup L_{x'}^\theta\subseteq\tilde\Omega\]
and verifying 

\

\begin{itemize}
\item [$(i)$] $|D \boldsymbol u|+|D \boldsymbol v|\leq \varepsilon$ in $Z_{x'}^{\varepsilon}$;
\item [$(ii)$] the $\mathbb R$-Lebesgue measure $|L_{x'}^\theta|\leq \theta$. 
\end{itemize}

\

\noindent Then there exist positive constants 
\[\bar \varepsilon=\bar \varepsilon(p,N, C_f,\Omega, \|D {\boldsymbol u}\|_{L^\infty(\Omega)}, \|D {\boldsymbol v}\|_{L^\infty(\Omega)})\]
and
\[\bar \theta=\bar\theta (p,N, C_f,\Omega, \|D {\boldsymbol u}\|_{L^\infty(\Omega)}, \|D {\boldsymbol v}\|_{L^\infty(\Omega)}),\]  such that for $0< \varepsilon\leq \bar \varepsilon$ and $0<\theta\leq \bar\theta$ it follows that 
\[\boldsymbol u\leq \boldsymbol v \quad \text{in }  \tilde\Omega.\]
\end{theorem}
\begin{proof}
First, let us introduce the vectorial function
\begin{equation}\label{eq:(20)1}
(\boldsymbol{u-v})^+:= \left( (u^1-v^1)^+,\ldots, (u^\ell-v^\ell)^+,\ldots, (u^N-v^N) \right).
\end{equation}

Since for all $\ell\in\{1,\ldots,N\}$ we have that $u^\ell\leq v^\ell$ on $\partial \Omega'$, it turns out that \eqref{eq:(20)1} is a good test function for equations \eqref{eq:(20)0}. If we subtract them, using \eqref{eq:lucio}, we obtain  
\begin{eqnarray}\label{eq:(20)2}
&&C_p \int_{\tilde\Omega}(|D {\boldsymbol u}|+|D \boldsymbol{v}|)^{p-2}|D (\boldsymbol{u-v})^+|^2 dx \nonumber \\
&& \leq \int_{\tilde\Omega} (|D \boldsymbol{u}|^{p-2}D \boldsymbol{u} - |D \boldsymbol{v}|^{p-2}D \boldsymbol{v}: D (\boldsymbol{u-v})^+)\, dx  \nonumber \\
&& \leq \int_{\tilde\Omega}(\boldsymbol{f}(x,\boldsymbol{u})-\boldsymbol{f}(x,\boldsymbol{v}))\cdot (\boldsymbol{u-v})^+\, dx.
\end{eqnarray}
The right hand side term of \eqref{eq:(20)2} can be arranged as follows
\begin{eqnarray}\label{eq:(20)21}\nonumber
&& \int_{\tilde\Omega}(\boldsymbol{f}(x,\boldsymbol{u})-\boldsymbol{f}(x,\boldsymbol{v}))\cdot (\boldsymbol{u-v})^+ \, dx=\\\nonumber &&\int_{\tilde\Omega} \sum_{\ell=1}^N \left[f^\ell(x,u^1,\ldots,u^N)-f^\ell(x,v^1,\ldots,v^N)\right] (u^\ell-v^\ell)^+\, dx   \nonumber \\
&& = \int_{\tilde\Omega} \sum_{\ell=1}^N \left[f^\ell(x,u^1,u^2,\ldots, u^N)-f^\ell(x,v^1,u^2\ldots, u^N)+f^\ell(x,v^1,u^2\ldots, u^N) \right. \nonumber \\
&& \qquad \qquad \qquad  \left.  -f^\ell(x,v^1,\ldots,v^N)\right] (u^\ell-v^\ell)^+ \,dx \\
&& = \int_{\tilde\Omega}\sum_{\ell=1}^N  \left[f^\ell(x,u^1,u^2,\ldots, u^N)-f^\ell(x,v^1,u^2\ldots, u^N)+f^\ell(x,v^1,u^2,\ldots, u^N) \right. \nonumber \\
&& -f^\ell(x,v^1,v^2,\ldots, u^N)+\cdots+f^\ell(x,v^1,v^2,\ldots, u^\ell,\ldots, u^N)\nonumber \\ &&\qquad \qquad \qquad-f^\ell(x,v^1,v^2,\ldots, v^\ell,\ldots, u^N) \nonumber \\
&& \qquad \vdots  \nonumber  \\
&& +\ldots + \left. f^\ell(x,v^1,v^2,\ldots, u^N)-f^\ell(x,v^1,v^2,\ldots, v^N) \right] (u^\ell-v^\ell)^+\, dx \nonumber.
\end{eqnarray}
Now, using the fact that $\boldsymbol{f}$ is a Lipschitz function with respect to the variable $t_j$, see $(i)$ of \eqref{hpf}, by \eqref{eq:(20)21} we have that
\begin{eqnarray}\label{eq:(20)22}
&& \int_{\tilde\Omega}(\boldsymbol{f}(x,\boldsymbol{u})-\boldsymbol{f}(x,\boldsymbol{v}))\cdot (\boldsymbol{u-v})^+\, dx \nonumber \\
&& \leq \int_{\tilde\Omega} \sum_{\ell=1}^N \left[   \frac{f^\ell(x,u^1,u^2,\ldots, u^N)-f^\ell(x,v^1,u^2\ldots, u^N)}{(u^1-v^1)^+} (u^1-v^1)^+ (u^\ell-v^\ell)^+ \right. \nonumber \\
&& + \frac{ f^\ell(x,v^1,u^2,\ldots, u^N)  -f^\ell(x,v^1,v^2,\ldots, u^N)}{(u^2-v^2)^+} (u^2-v^2)^+ (u^\ell-v^\ell)^+ \nonumber \\
&& \vdots   \\
&&+  \frac{f^\ell(x,v^1,v^2,\ldots, u^\ell,\ldots, u^N)-f^\ell(x,v^1,v^2,\ldots, v^\ell,\ldots, u^N)}{ (u^\ell-v^\ell)^+}{[(u^\ell-v^\ell)^+]^2} \nonumber  \\
 && \vdots \nonumber  \\
&& + \left. \frac{f^\ell(x,v^1,v^2,\ldots, u^N)-f^\ell(x,v^1,v^2,\ldots, v^N)}{(u^N-v^N)^+} (u^N-v^N)^+ (u^\ell-v^\ell)^+ \right]\, dx \nonumber \\
&& \leq C_f \int_{\tilde\Omega} \sum_{\ell=1}^N  \sum_{j=1}^N   (u^j-v^j)^+ (u^\ell-v^\ell)^+\, dx \leq N C_f \int_{\tilde\Omega} \sum_{\ell=1}^N[(u^\ell-v^\ell)^+]^2 dx, \nonumber
\end{eqnarray}
where in the last inequality we exploit Young's inequality  and where $C_f=\displaystyle\max_\ell\{L_\ell\}$, see \eqref{hpf},  denotes the Lipschitz constant. So, by \eqref{eq:(20)2} and \eqref{eq:(20)22} we get that 
\begin{eqnarray}\label{eq:(20)23}
C_p \int_{\tilde\Omega}(|D {\boldsymbol u}|+|D \boldsymbol{v}|)^{p-2}|D (\boldsymbol{u-v})^+|^2 dx &&\leq NC_f \int_{\tilde\Omega} \sum_{\ell=1}^N[(u^\ell-v^\ell)^+]^2 dx
\nonumber \\
&&=NC_f\int_{\Omega'}dx'\int_{I_{x'}}\sum_{\ell=1}^N[(u^\ell-v^\ell)^+]^2dx_n,
\end{eqnarray}
by Fubini's theorem. Moreover, since for all $x'\in \Omega'$ and for all $\ell$, the function $(u^\ell- v^\ell)^+=0$ on $\partial I_{x'} $,  there exists $a\in \mathbb R$ such that
\begin{eqnarray}\label{eq:(20)3}
&&\Big |(u^\ell- v^\ell)^+(x',t)\Big | 
\leq \int_a^t |\partial_{s}(u^\ell- v^\ell)^+|(x',s)ds \\\nonumber
&&\leq\int_{I_{x'}} |\partial_{s}(u^\ell- v^\ell)^+|(x',s)ds.
\end{eqnarray}
Introducing \eqref{eq:(20)3} in \eqref{eq:(20)23}, we get
\begin{eqnarray}\label{eq:(20)4}
\\\nonumber
&&C_p \int_{\tilde\Omega}(|D {\boldsymbol u}|+|D \boldsymbol{v}|)^{p-2}|D (\boldsymbol{u-v})^+|^2 dx  \\\nonumber
&&\leq
NC_f\int_{\Omega'}dx'\int_{I_{x'}}\sum_{\ell=1}^N|(u^\ell-v^\ell)^+|^2dx_n\\\nonumber
&&\leq
NC_f\int_{\Omega'}dx'\int_{I_{x'}}\sum_{\ell=1}^N \left (\int_{I_{x'}} |\partial_{x_n}(u^\ell- v^\ell)^+|dx_n\right)^2dx_n\\\nonumber
&& \leq
NC_f\int_{\Omega'}dx'\int_{I_{x'}} \sum_{\ell=1}^N\left(\int_{Z_{x'}^{\varepsilon} \cup L_{x'}^\theta} |\nabla (u^\ell- v^\ell)^+|dx_n\right)^2dx_n\\\nonumber
&& \leq C\int_{\Omega'}dx'\sum_{\ell=1}^N\left(|Z_{x'}^{\varepsilon}|\int_{Z_{x'}^{\varepsilon}}|\nabla (u^\ell- v^\ell)^+|^2dx_n+| L_{x'}^\theta|\int_{ L_{x'}^\theta}|\nabla (u^\ell- v^\ell)^+|^2dx_n\right),
\end{eqnarray}
with $C=C(N,C_f,\Omega)$ a positive constant. In the last inequality we used that fact that for all $a,b\geq 0$ it follows $(a+b)^2\leq 2(a^2+b^2)$ and  H\"older inequality.  By Fubini's theorem, from \eqref{eq:(20)4} we finally deduce (recall that by hypothesis $|L_{x'}^\theta|\leq \theta$, for  $x'\in \Omega'$)
\begin{eqnarray}\nonumber
\\\nonumber
&&C_p \int_{\tilde\Omega}(|D {\boldsymbol u}|+|D \boldsymbol{v}|)^{p-2}|D (\boldsymbol{u-v})^+|^2 dx  \\\nonumber
&&\leq C\sum_{\ell=1}^N\int_{\Omega'\times Z_{x'}^{\varepsilon}}|\nabla (u^\ell- v^\ell)^+|^2dx\\\nonumber
&&+C| L_{x'}^\theta|\sum_{\ell=1}^N\int_{\Omega'\times L_{x'}^\theta}|\nabla (u^\ell- v^\ell)^+|^2dx\\\nonumber
&&\leq C  \sup_{\Omega'\times Z_{x'}^{\varepsilon}}(|D \boldsymbol{u}|+|D \boldsymbol{v}|)^{2-p}\int_{\tilde\Omega}(|D \boldsymbol{u}|+|D \boldsymbol{v}|)^{p-2}|D (\boldsymbol{u-v})^+|^2dx\\\nonumber
&&+C | L_{x'}^\theta| \sup_{\Omega}(|D \boldsymbol{u}|+|D \boldsymbol{v}|)^{2-p}\int_{\tilde\Omega}(|D \boldsymbol{u}|+|D \boldsymbol{v}|)^{p-2}|D (\boldsymbol{u-v})^+|^2dx\\\nonumber
 && \leq C(\varepsilon^{2-p}+\theta)\int_{\tilde\Omega}(|D \boldsymbol{u}|+|D \boldsymbol{v}|)^{p-2}|D (\boldsymbol{u-v})^+|^2dx,
\end{eqnarray}
since $p\leq 2$ and therefore
\begin{equation}\label{eq:(20)5}
\int_{\tilde\Omega}(|D {\boldsymbol u}|+|D \boldsymbol{v}|)^{p-2}|D (\boldsymbol{u-v})^+|^2 dx  \leq C(\varepsilon^{2-p}+\theta)\int_{\tilde\Omega}(|D \boldsymbol{u}|+|D \boldsymbol{v}|)^{p-2}|D (\boldsymbol{u-v})^+|^2dx,
\end{equation}
where $C=C(p,N, C_f,\Omega, \|D {\boldsymbol u}\|_{L^\infty(\Omega)}, \|D {\boldsymbol v}\|_{L^\infty(\Omega)})$ is a positive constant. 
Then there exist $\bar \varepsilon, \bar \theta$ depending on $p,N, C_f,\Omega, \|D {\boldsymbol u}\|_{L^\infty(\Omega)}, \|D {\boldsymbol v}\|_{L^\infty(\Omega)}$, such that for all  $0\leq \varepsilon\leq \bar \varepsilon$ and $0\leq \theta\leq \bar\theta$ it is satisfied
\[C(\varepsilon^{2-p}+\theta)< \frac{1}{2}.\]
Putting the last inequality in  \eqref{eq:(20)5}, we obtain the thesis.
\end{proof}

Next, we deal with a first version of the strong comparison principle, under the assumption that compared functions are strictly separated on the boundary of the domain.

\begin{theorem}[Strong Comparison Principle 1]\label{thm:boundarysep}
Let $\boldsymbol{u, v} \in C^{1}(\overline\Omega)$ satisfying  \eqref{ineq} and let us assume that  either $\boldsymbol{u}$ or $\boldsymbol{v}$ is a weak solution to the problem \eqref{system}.
Suppose that $\boldsymbol{f}$ satisfies~\eqref{hpf} and that  
\begin{equation}\label{eq:raffmor}
\boldsymbol{u} \leq \boldsymbol{v}\quad \text{in }\Omega,
\end{equation} namely $u^\ell\leq v^\ell$ in $\Omega$ for $\ell=1,\ldots,N$. 

Then, if $\boldsymbol{u}<\boldsymbol{v}$ on $\partial \Omega$, it holds that
$$
\boldsymbol{u}<\boldsymbol{v}\quad \text{in }  \Omega.
$$ 
\end{theorem}

\begin{proof}
Without loss of generality, suppose that $\boldsymbol{u}$ solves \eqref{system}. Now, consider the set where $u^\ell$ and $v^\ell$ may coincide, namely 
$$
\mathcal{C}_{u,v}^\ell=\{x\in \Omega \, | \, u^\ell(x)=v^\ell(x) \}.
$$
We want to prove that $\mathcal{C}_{u,v}^\ell= \emptyset$ for any $\ell=1,\ldots,N$. Assume by contradiction that there exists some $\ell\in \{1,\ldots, N\}$ such that $\mathcal{C}_{u,v}^\ell \neq \emptyset$. 
In particular, take $\mathcal{L}$ the set of indexes such that $\mathcal{C}_{u,v}^\ell \neq \emptyset$. 
By our assumptions $\mathcal{C}_{u,v}^\ell$ is a closed set and it lies far away from the boundary $\partial \Omega$ where $\boldsymbol{u}<\boldsymbol{v}$. This implies, in particular, that  $\partial \mathcal{C}_{u,v}^{\ell} \neq \emptyset$.

Let $\varepsilon>0$ small enough and define the set 
$$
(\mathcal{C}_{u,v}^\ell)^\varepsilon = \{x \in \Omega \,\, | \,\, dist(x,\mathcal{C}_{u,v}^\ell)<\varepsilon \}.
$$
Since  $\partial (\mathcal{C}_{u,v}^\ell)^\varepsilon$ is a compact set,
then there exists $\tau>0$ such that $u^\ell+\tau<v^\ell$ on~$\partial (\mathcal{C}_{u,v}^\ell)^\varepsilon$.  Next, define the function $w_\tau:\overline{\Omega}\to[0,\infty)^N$ such that
$$
w^\ell_\tau= \left\{\begin{array}{ll}(u^\ell-v^\ell+\tau)^+ &\text{in } \displaystyle{ (\mathcal{C}_{u,v}^\ell)^\varepsilon} \vspace{0.2cm} \\
0 &\text{in }\displaystyle{ \overline{\Omega}\setminus(\mathcal{C}_{u,v}^\ell)^\varepsilon}. \end{array}\right.
$$
Observe that  
 $u^\ell+\tau\geq v^\ell$ in $(\mathcal{C}_{u,v}^\ell)^{\varepsilon}$ and indeed that $u^\ell+\tau > v^\ell$ in $\mathcal{C}_{u,v}^\ell= \emptyset$. 
 Moreover  we set  $w^\ell_\tau=0$ in $\Omega$ for  $\ell \notin \mathcal{L}$. Since $u^\ell+\tau<v^\ell$ on  $\partial (\mathcal{C}_{u,v}^\ell)^\varepsilon$, then we infer that $\boldsymbol{w_\tau} \in W_0^{1,p}(\Omega)$. Moreover,
$$
\nabla w^\ell_\tau= \left\{\begin{array}{ll} \nabla u^\ell-\nabla v^\ell  &\text{where } \displaystyle{ w_\tau>0} \vspace{0.2cm} \\
0 &\text{otherwise. }\end{array}\right.
$$

Introducing $\boldsymbol{w_\tau}$ as a test function in \eqref{weak}, we obtain that
\beq\label{uno}
\displaystyle{ \int_{\Omega} |D {\boldsymbol u}|^{p-2}D {\boldsymbol u} : D \boldsymbol{w_\tau}\, dx   = \int_{\Omega} \boldsymbol{f}(x,\boldsymbol{u}) \cdot \boldsymbol{w_\tau}\, dx = \int_{\bigcup_{\ell\in\mathcal{L}} (\mathcal{C}_{u,v}^\ell)^\varepsilon} \boldsymbol{f}(x,\boldsymbol{u}) \cdot \boldsymbol{w_\tau}}\, dx.
\eeq
Since $f$ fulfills \eqref{hpf} (in particular $(iii)$ of \eqref{hpf}),  using the assumption \eqref{eq:raffmor}  (i.e. $\boldsymbol{u} \leq \boldsymbol{v} $ in $\Omega$), \eqref{ineq} and that $\boldsymbol{u}$ is solution of \eqref{system} (namely the  inequality ${\boldsymbol f}(x, {\boldsymbol v})\leq -{\bf\Delta}_p {\boldsymbol v}$  holds in $\Omega$),  we obtain
\beq\label{due}
\displaystyle{
 \int_{\bigcup_{\ell\in\mathcal{L}} (\mathcal{C}_{u,v}^\ell)^\varepsilon} \boldsymbol{f}(x,\boldsymbol{u}) \cdot \boldsymbol{w_\tau} \, dx \leq \int_{\bigcup_{\ell\in\mathcal{L}} (\mathcal{C}_{u,v}^\ell)^\varepsilon} \boldsymbol{f}(x,\boldsymbol{v}) \cdot \boldsymbol{w_\tau} \, dx  \leq \int_{\Omega} |D \boldsymbol{v}|^{p-2}D \boldsymbol{v} : D \boldsymbol{w_\tau}} \, dx.
\eeq
We point out that, to get the first inequality in \eqref{due} we follow the computations in~\eqref{eq:(20)21}.

By \eqref{uno} and \eqref{due}, one has
\begin{eqnarray}
\label{tre}
&&\int_{\Omega} \left( |D {\boldsymbol u}|^{p-2}D {\boldsymbol u}-|D {\boldsymbol u}|^{p-2}D \boldsymbol{v} \right): D \boldsymbol{w_\tau}\, dx\\\nonumber
&&=\int_{w_\tau>0} \left( |D {\boldsymbol u}|^{p-2}D {\boldsymbol u}-|D {\boldsymbol u}|^{p-2}D \boldsymbol{v} \right): (D {\boldsymbol u}-D \boldsymbol{v}) \, dx\leq 0.
\end{eqnarray}
Applying in \eqref{tre}  the inequality \eqref{eq:lucio}, we obtain   that 
$$
\displaystyle{ \int_{w_\tau>0} (|D \boldsymbol{u}|+|D \boldsymbol{v}|)^{p-2}|D \boldsymbol{u}-D \boldsymbol{v}|^2  \, dx\leq 0}.
$$

Immediately, we have that $\boldsymbol{u-v}$ is a constant where $\boldsymbol{w_\tau}>0$. By continuity of $\boldsymbol{w_\tau}$, $w_\tau^\ell$ must vanish on $\partial (\mathcal{C}^\ell_{u,v})^\varepsilon$ for $\ell\in \mathcal{L}$. Then $w^\ell_\tau=0$ in $(\mathcal{C}^\ell_{u,v})^\varepsilon$, namely $u^\ell<v^\ell$  in~$(\mathcal{C}^\ell_{u,v})^\varepsilon$.
This is a contradiction with the fact that $(\mathcal{C}_{u,v}^\ell)^{\varepsilon} \supset \mathcal{C}^\ell_{u,v} \neq \emptyset$ for~$\ell\in \mathcal{L}$. Therefore we get the thesis.
\end{proof}

Finally, we present a second version of the strong comparison principle. Instead of a separation from the boundary, we will assume that inequality \eqref{ineq} is satisfied strictly. This will allow us to derive a strict relation between $\boldsymbol{u,v}$. Notice that this condition will be provided in the main theorems by the strict monotonicity of $\boldsymbol{f}$. In order to work with sufficiently regular functions, we will consider far away from the critical value set $\mathcal{Z}_u$.

\begin{theorem}\label{thm:strong}(Strong Comparison Principle 2)
Let $C\Subset \Omega\setminus \mathcal Z_{u}$ be a connected set, where $\mathcal Z_{u}$ is defined in \eqref{eq:critraf}. Let us suppose that ${\boldsymbol u}$ is a $C^1(\Omega)$ weak solution to  \eqref{system} and  ${\boldsymbol v}\in C^2{(C)}$ such that   ${\boldsymbol u}\leq {\boldsymbol v}$ in $C$ and 
\begin{equation}\label{eq:regvit}
-{\bf\Delta}_p {\boldsymbol u} -{\boldsymbol f}(x,{\boldsymbol u}) < -{\bf\Delta}_p {\boldsymbol v} -{\boldsymbol f}(x,{\boldsymbol v})\quad \text{in } C,\end{equation}
where ${\boldsymbol f}$ verifies \eqref{hpf}.

Then 
\[\boldsymbol{u}<\boldsymbol{v}\quad \text{in }  C.\]
\end{theorem}
\begin{proof}  
 Let us fix $x_0\in C$ and let us choose $R>0$ small enough such that $B_{4R}=B_{4R}(x_0)\subset C$. Moreover we also have $D {\boldsymbol u}(x)\neq 0$ in $B_{4R}$.

Obviously, by the assumptions, $\boldsymbol{u} \leq \boldsymbol{v} $ on $\partial B_{3R}$. For $\varepsilon >0$ let us define $\varphi_\varepsilon \in C^{\infty}(B_{3R})$ given by
\begin{equation}\label{eq:ted0}
\varphi_\varepsilon:=
\begin{cases}
1 & \text{in}\,\, B_R \\
1-\varepsilon & \text{in}\,\,  B_{3R}\setminus B_{2R},
\end{cases} 
\end{equation}
such that $|\nabla \varphi_\varepsilon|, |\Delta \varphi_\varepsilon|\leq C \varepsilon$.

By standard regularity, $\boldsymbol{u}\in C^2(B_{4R})$, since $D {\boldsymbol u}(x)\neq 0$.  Hence let $\boldsymbol{u}_\varepsilon:=\boldsymbol{u}\varphi_{\varepsilon} \in C^2(B_{3R},\mathbb{R}^N)$. Therefore
\begin{eqnarray}\label{eq:ted00}
&&-\Delta_p \boldsymbol{u}_\varepsilon=-\operatorname{div}(|D(\boldsymbol{u}\varphi_\varepsilon)|^{p-2}D(\boldsymbol{u}\varphi_\varepsilon))\\\nonumber
&&= -|D {\boldsymbol u}_\varepsilon|^{p-2}\Delta \boldsymbol{u}_{\varepsilon}-(p-2)|D \boldsymbol{u}_\varepsilon|^{p-4}\left(
\begin{array}{c}
\sum_{m=1}^N\sum_{\ell,j=1}^n\partial_{x_\ell}u^1_\varepsilon\partial_{x_j}u^m_\varepsilon\partial^2_{x_ix_j}u^m_\varepsilon\\
\vdots\\
\sum_{m=1}^N\sum_{\ell,j=1}^n\partial_{x_\ell}u^l_\varepsilon\partial_{x_j}u^m_\varepsilon\partial^2_{x_ix_j}u^m_\varepsilon\\
\vdots\\
\sum_{m=1}^N\sum_{\ell,j=1}^n\partial_{x_\ell}u^N_\varepsilon\partial_{x_j}u^m_\varepsilon\partial^2_{x_ix_j}u^m_\varepsilon
\end{array}\right)
\\\nonumber
&&:=P_1+ P_2,
\end{eqnarray}
with $\ell=1,\ldots,N.$
We have that
\begin{equation}\label{eq:ted1}
\nabla u_{\varepsilon}^\ell=\nabla u^\ell \varphi_\varepsilon+ u^\ell\nabla \varphi_\varepsilon .
\end{equation} Using Taylor formula, for $\ell=1, \ldots, N$, we infer that
\begin{eqnarray}\label{eq:ted2}
|\nabla u^\ell_\varepsilon|^2 &=& \varphi_\varepsilon^2|\nabla u^\ell|^2+2\varphi_\varepsilon u^\ell \nabla u^\ell \cdot \nabla \varphi_\varepsilon+ O(u^\ell|\nabla \varphi_\varepsilon|)\\\nonumber
&=& \varphi_\varepsilon^2|\nabla u^\ell|^2 + O(\varepsilon^2),
\end{eqnarray}
since definition  \eqref{eq:ted0}. Therefore, since $D {\boldsymbol u}(x)\neq 0$ in $B_{4R}$,  by~\eqref{eq:ted2} we deduce
\begin{eqnarray}\label{eq:ted3}
\quad |D {\boldsymbol u}_\varepsilon|^{p-2}&=&\left(\sum_{\ell=1}^N |\nabla u_\varepsilon^\ell|^2\right)^{\frac{p-2}{2}}\\\nonumber&=&\varphi_\varepsilon^{p-2}|D {\boldsymbol u}|^{p-2}
+\frac{p-2}{2} \varphi_\varepsilon^{p-4} \left(\sum_{\ell=1}^N |\nabla u^\ell|^2\right)^{\frac{p-4}{2}}O(\varepsilon^2)+O (\varepsilon^4)
\\\nonumber
&=&\varphi_\varepsilon^{p-2}|D \boldsymbol{u}|^{p-2} + O (\varepsilon^2).
\end{eqnarray}
We estimate now the two terms of 
\eqref{eq:ted00}:
\begin{eqnarray}\nonumber
&&P_1:=-|D {\boldsymbol u}_\varepsilon|^{p-2}\boldsymbol{\Delta} \boldsymbol{u}_\varepsilon=-|D {\boldsymbol u}_\varepsilon|^{p-2}\left(
\begin{array}{c}
2\nabla u^1\nabla \varphi_\varepsilon + \Delta u^1\varphi_\varepsilon +u^1\Delta \varphi_\varepsilon\\
\vdots\\
2\nabla u^\ell\nabla \varphi_\varepsilon + \Delta u^\ell\varphi_\varepsilon +u^\ell\Delta \varphi_\varepsilon\\
\vdots \\
2\nabla u^N\nabla \varphi_\varepsilon + \Delta u^N\varphi_\varepsilon +u^N\Delta \varphi_\varepsilon
\end{array}
\right)\\\nonumber
&&=-|D {\boldsymbol u}_\varepsilon|^{p-2}\varphi_\varepsilon \boldsymbol{\Delta} \boldsymbol{u}+
\left (
\begin{array}{c}
O(\varepsilon)\\
\vdots\\
O(\varepsilon)\\
\vdots \\
O(\varepsilon)
\end{array}\right)=-|D \boldsymbol{u}|^{p-2}\varphi_\varepsilon^{p-1}\boldsymbol{\Delta} \boldsymbol{u}+\left (
\begin{array}{c}
O(\varepsilon)\\
\vdots\\
O(\varepsilon)\\
\vdots \\
O(\varepsilon)
\end{array}\right),
\end{eqnarray}
where in the last inequality we used \eqref{eq:ted3}. Using definition \eqref{eq:ted0}, we also have
\begin{eqnarray}\nonumber
\partial_{x_i}u^m_\varepsilon&=&\partial_{x_i}u^m\varphi_\varepsilon + O(\varepsilon)\\\nonumber
\partial^2_{x_ix_j}u^m_\varepsilon&=&\partial^2_{x_i x_j}u^m\varphi_\varepsilon + O(\varepsilon).
\end{eqnarray}
Then we deduce that
\begin{eqnarray}\label{eq:ted4}\\\nonumber
P_2&:=&-(p-2)|D \boldsymbol{u}_\varepsilon|^{p-4}\left(
\begin{array}{c}
\sum_{m=1}^N\sum_{i,j=1}^n\partial_{x_i}u^1_\varepsilon\partial_{x_j}u^m_\varepsilon\partial^2_{x_ix_j}u^m_\varepsilon\\
\vdots\\
\sum_{m=1}^N\sum_{i,j=1}^n\partial_{x_i}u^l_\varepsilon\partial_{x_j}u^m_\varepsilon\partial^2_{x_ix_j}u^m_\varepsilon\\
\vdots\\
\sum_{m=1}^N\sum_{i,j=1}^n\partial_{x_i}u^N_\varepsilon\partial_{x_j}u^m_\varepsilon\partial^2_{x_ix_j}u^m_\varepsilon
\end{array}\right)
\\\nonumber
\\\nonumber
&=&
-(p-2)|D \boldsymbol{u}_\varepsilon|^{p-4}\left(
\begin{array}{c}
\sum_{m=1}^N\sum_{i,j=1}^n\varphi_\varepsilon^3\partial_{x_i}u^1\partial_{x_j}u^m\partial^2_{x_ix_j}u^m+O(\varepsilon)\\
\vdots\\
\sum_{m=1}^N\sum_{i,j=1}^n\varphi_\varepsilon^3\partial_{x_i}u^l\partial_{x_j}u^m\partial^2_{x_ix_j}u^m+O(\varepsilon)\\
\vdots\\
\sum_{m=1}^N\sum_{i,j=1}^n\varphi_\varepsilon^3\partial_{x_i}u^N\partial_{x_j}u^m\partial^2_{x_ix_j}u^m+O(\varepsilon)
\end{array}\right)\\\nonumber
&=&-(p-2)\varphi_\varepsilon^{p-4}|D \boldsymbol{u}|^{p-4}\left(
\begin{array}{c}
\sum_{m=1}^N\sum_{i,j=1}^n\varphi_\varepsilon^3\partial_{x_i}u^1\partial_{x_j}u^m\partial^2_{x_ix_j}u^m+O(\varepsilon)\\
\vdots\\
\sum_{m=1}^N\sum_{i,j=1}^n\varphi_\varepsilon^3\partial_{x_i}u^l\partial_{x_j}u^m\partial^2_{x_ix_j}u^m+O(\varepsilon)\\
\vdots\\
\sum_{m=1}^N\sum_{i,j=1}^n\varphi_\varepsilon^3\partial_{x_i}u^N\partial_{x_j}u^m\partial^2_{x_ix_j}u^m+O(\varepsilon)
\end{array}\right).
\end{eqnarray}
Hence, we can conclude that
\[-\boldsymbol{\Delta}_p \boldsymbol{u}_\varepsilon=P_1+P_2=\varphi_\varepsilon^{p-1}(-\boldsymbol{\Delta}_p \boldsymbol{u}) + O(\varepsilon).\]
Then for suitable small $\varepsilon$
\begin{eqnarray}\label{eq:ted5}
&&-\boldsymbol{\Delta}_p \boldsymbol{u}_\varepsilon -\boldsymbol{f}(x,\boldsymbol{u}_\varepsilon)=\varphi_\varepsilon^{p-1}\boldsymbol{f}(x,\boldsymbol{u})-\boldsymbol{f}(x,\boldsymbol{u}_\varepsilon)+ O(\varepsilon)\\\nonumber
&&=\boldsymbol{f}(x,\boldsymbol{u})(\varphi_\varepsilon^{p-1}-1) +\boldsymbol{f}(x,\boldsymbol{u})-\boldsymbol{f}(x,\boldsymbol{u}_\varepsilon) + (-\boldsymbol{\Delta}_p \boldsymbol{u}-\boldsymbol{f}(x,\boldsymbol{u}))+ O(\varepsilon)\\\nonumber
&&<-\boldsymbol{\Delta}_p \boldsymbol{v}-\boldsymbol{f}(x,\boldsymbol{v}) \, \, \mbox{ in } B_{3R},
\end{eqnarray}
where the last inequality follows since \eqref{eq:regvit} is in force and because $\boldsymbol{u,v}\in C^2(B_{4R})$ and  ${\boldsymbol f}$ is a continuous vectorial function. 
By \eqref{eq:ted5} and taking into account that $\boldsymbol{u}_\varepsilon< \boldsymbol{v}$ in $\partial B_{3R}$, then we can apply Theorem~\ref{thm:boundarysep} to state that $\boldsymbol{u}_\varepsilon<\boldsymbol{v}$ in $B_{3R}$. Since $\boldsymbol{u}_\varepsilon \equiv \boldsymbol{u}$ in $B_R$, the conclusion follows by the arbitrarily of $x_0$.
\end{proof}

\begin{remark}
In the previous result we have assumed that $\boldsymbol{v}\in C^2(C)$. This condition is automatically satisfied if the $C$ is far away from the critical set $\mathcal{Z}_v$. As a consequence, the conclusion is true if one considers $C\Subset \Omega \setminus (\mathcal{Z}_u \cup \mathcal{Z}_v)$. Actually, Theorem~\ref{thm:strong} will be used to prove Theorems~\ref{thm:main1},\ref{thm:main2} verifying these circumstances.
\end{remark}

\section{Main results}\label{mainproof}
\setcounter{equation}{0}

This section is devoted to prove Theorems~\ref{thm:main1} and Theorem~\ref{thm:main2} by using of the moving-planes technique, once weak and strong comparison principles have been deduced in the previous section.

In order to do it, we will introduce some notations and assumptions that we are going to use in the following. For a real number $\lambda$, let 
\begin{equation}\label{eq:whawha}\Omega_\lambda=\Omega \cap \{x_n\leq \lambda\},\end{equation}
\[x_\lambda =R_{\lambda}(x):=(x_1,\ldots,2\lambda -x_n),\]
which is the point reflected through the hyperplane
\[T_\lambda =\{x\in \mathbb R^n\, :\, x_n=\lambda \}.\]  
Also set 
\begin{equation}\label{eq:a}
a=\inf_{x\in \Omega}\{x_n\},
\end{equation}
and 
\[\boldsymbol{u}_{\lambda}=\boldsymbol{u}(x_\lambda)=\boldsymbol{u}(x_1,\ldots, x_{n-1}, 2\lambda -x_n).\]

\

In the following results, we shall deduce some properties concerning the critical set $\mathcal{Z}_u$ defined as \eqref{eq:critraf}.

First of all, we point out that by standard elliptic regularity arguments, a $C^1(\Omega)$ solution $\boldsymbol{u}$ of \eqref{system} with $\boldsymbol{f}$ satisfying \eqref{hpf} is indeed smooth in $\Omega\setminus \mathcal{Z}_u$, see \cite{GT}. Consequently, the distributional derivatives of $|D \boldsymbol{u}|^{p-2}u^\ell_{j}$ coincide with the classical ones.

By using the Stampacchia's Theorem, (see  for instance \cite[Lemma 7.7]{GT}, \cite[Theorem~1.56]{Troia}), indeed the generalized derivatives of $|D \boldsymbol{u}|^{p-2} u^\ell_{ij}$ are zero almost everywhere in $\left\{u^\ell_{j}=0 \right\}$. Moreover, in the following we shall use the following notion of distributional derivative
$$
 u^\ell_{i j}= \left\{\begin{array}{ll}  u^\ell_{i j}  &\text{in } \Omega \setminus \mathcal{Z}_u \vspace{0.2cm} \\
0 &\text{in } \mathcal{Z}_u. \end{array}\right.
$$
In addition, we will adopt this convention for the gradients $D u^\ell_i$.

\

Next, we shall stress some properties concerning $|D \boldsymbol{u}|^{p-2} D^2 \boldsymbol{u}$ and $|D \boldsymbol{u}|^{p-1}$ under the assumption $|D \boldsymbol{u}|^{p-2} D \boldsymbol{u} \in W_{loc}^{1,2}(\Omega)$. Recall that this condition is deduced in \cite{Cma} under some general hypotheses on $\boldsymbol{f}$ and $p$. Although, these results appear implicitly in the aforementioned reference, let us deduce them to apply in the rest of the section.
\begin{lemma}\label{lemmaL2loc}
Let $\boldsymbol{u}\in C^1(\overline \Omega)$ such that $|D \boldsymbol{u}|^{p-2}D \boldsymbol{u} \in W_{loc}^{1,2}(\Omega) $. Then
\begin{itemize}
\item[(i)] $|D \boldsymbol{u}|^{p-2} D^2 \boldsymbol{u} \in L_{loc}^{2}(\Omega)$ for $1<p<3$,
\item[(ii)] $|D \boldsymbol{u}|^{p-1} \in W_{loc}^{1,2}(\Omega)$.
\end{itemize}
\end{lemma}

\begin{proof}

Let us differentiate $|D \boldsymbol{u}|^{p-2} u^{\ell}_j$ with respect to $x_i$ with $i,j\in \{ 1,\ldots, n\}$ and $\ell \in \{1,\ldots,N\}$, namely
\begin{equation}\label{deri}
\frac{\partial}{\partial x_i}\left( |D \boldsymbol{u}|^{p-2} u^\ell_{j} \right) = |D \boldsymbol{u}|^{p-2}  {u}^\ell_{ij} + (p-2) |D \boldsymbol{u}|^{p-4}(D \boldsymbol{u} : {D}\boldsymbol{u}_i )  {u}^\ell_{j}:=h(x).
\end{equation}

Since $|D \boldsymbol{u}|^{p-2} D  \boldsymbol{u} \in W^{1,2}_{loc}$, then  $h\in L^2_{loc}(\Omega)$. By Young's inequality, one has
\begin{align}\label{deri1}
&|D \boldsymbol{u}|^{2(p-2)}  ({u}^\ell_{ij})^2
 \nonumber \\& =
 h^2+(2-p)^2 |D \boldsymbol{u}|^{2(p-4)}(D \boldsymbol{u} : {D}\boldsymbol{u}_i )^2  ({u}^\ell_{j})^2+ 
 2h(2-p)|D \boldsymbol{u}|^{p-4}(D \boldsymbol{u} : {D}\boldsymbol{u}_i )  {u}^\ell_{j}
\nonumber \\& \leq C_\varepsilon h^2+ (1+\varepsilon^2)(2-p)^2 |D \boldsymbol{u}|^{2(p-4)}(D \boldsymbol{u} : {D}\boldsymbol{u}_i )^2  ({u}^\ell_{j})^2 \nonumber \\
& \leq  C_\varepsilon h^2 + (1+\varepsilon^2)(2-p)^2 |D \boldsymbol{u}|^{2(p-3)}|D \boldsymbol{u}_i|^2  ({u}^\ell_{j})^2.
\end{align}
Adding the previous expressions with respect to $i=1,\ldots,n$, we obtain
\begin{align}\label{deri2}
|D \boldsymbol{u}|^{2(p-2)}  |\nabla u_j^\ell|^2 
 & =\nonumber \sum_{i=1}^n|D \boldsymbol{u}|^{2(p-2)}  (u^\ell_{ij})^2 \\
& \leq  n C_\varepsilon h^2 + (1+\varepsilon^2)(2-p)^2 |D \boldsymbol{u}|^{2(p-3)}(u^\ell_{j})^2 \sum_{i=1}^n |D \boldsymbol{u}_i|^2 \\
& \leq  n C_\varepsilon h^2 + (1+\varepsilon^2)(2-p)^2 |D \boldsymbol{u}|^{2(p-3)}(u^\ell_{j})^2 |D^2 \boldsymbol{u}|^2. \nonumber
\end{align}
Now, we add \eqref{deri2} with respect to $j=1,\ldots,n$,
\begin{align}\nonumber
|D \boldsymbol{u}|^{2(p-2)}  |D^2 u^\ell|^2 & = \sum_{j=1}^n |D \boldsymbol{u}|^{2(p-2)}  |\nabla u_j^\ell|^2   \\\nonumber
& \leq  n^2C_\varepsilon h^2 + (1+\varepsilon^2)(2-p)^2 |D \boldsymbol{u}|^{2(p-3)} |D^2 \boldsymbol{u}|^2 |\nabla u^{\ell}|^2. 
\end{align}

Putting together all the components, we arrived at
\begin{align}\nonumber
 |D \boldsymbol{u}|^{2(p-2)}  |D^2 \boldsymbol{u}|^2  & =  
 \sum_{\ell=1}^N  |D \boldsymbol{u}|^{2(p-2)}  |D^2 u^\ell|^2 
  \nonumber \\\nonumber
& \leq  Nn^2C_\varepsilon h^2 + (1+\varepsilon^2)(2-p)^2 |D \boldsymbol{u}|^{2(p-3)} |D^2 \boldsymbol{u}|^2 |D \boldsymbol{u}|^2,
\end{align}
equivalent to
$$
(1-(1+\varepsilon^2)(2-p)^2) |D \boldsymbol{u}|^{2(p-2)}  |D^2 \boldsymbol{u}|^2 \leq C_\varepsilon h^2,
$$
which completes the proof of $(i)$.

By differentiating $|D \boldsymbol{u}|^{p-1}$ with respect to $x_i$, one gets 

\begin{equation}\nonumber\frac{\partial}{\partial x_i}\left( |D \boldsymbol{u}|^{p-1} \right) = (p-1) |D \boldsymbol{u}|^{p-3}(D \boldsymbol{u} : {D}\boldsymbol{u}_i )  \leq (p-1) | D \boldsymbol{u}|^{p-2}| |D^2 \boldsymbol{u} |,
\end{equation}
which gives us to $(ii)$ by taking into account $(i)$.
\end{proof}

Due to the positivity of $\boldsymbol{f}$, it is proved that $\mathcal{Z}_u$ is a zero measure set and $\Omega \setminus \mathcal{Z}_u$ is connected.
\begin{proposition}\label{lemmaZuzero}
Let ${\boldsymbol u} \in C^1(\overline{\Omega})$ be a weak solution to \eqref{system} such that $|D{\boldsymbol u}|^{p-2}D{\boldsymbol u} \in W_{loc}^{1,2}(\Omega) $. If ${\boldsymbol f}$ is a positive function in $\Omega$, then the critical set $\mathcal{Z}_u$ has zero Lebesgue measure.
\end{proposition}
\begin{proof}

Consider ${\boldsymbol \varphi} \in C^{\infty}_c(\Omega)$, then for $\varepsilon>0$, then
$$
{\boldsymbol \varphi} \frac{|D {\boldsymbol u}|^{p-1}}{\varepsilon+|D {\boldsymbol u}|^{p-1}} \in W_0^{1,p}(\Omega).
$$
Therefore, we can introduce it as a function test in \eqref{weak} to obtain that
\begin{align}\label{eq:lemmaZuzero1}\nonumber
\int_{\Omega} \frac{|D {\boldsymbol u}|^{p-1}}{\varepsilon+|D {\boldsymbol u}|^{p-1}} {\boldsymbol \varphi} \cdot {\boldsymbol f }\, dx &= \int_{\Omega} |D {\boldsymbol u}|^{p-2} D {\boldsymbol u} : D \left( {\boldsymbol \varphi} \frac{|D {\boldsymbol u}|^{p-1}}{\varepsilon+|D {\boldsymbol u}|^{p-1}} \right)\, dx \vspace{0.3cm} \\
& = \int_{\Omega} |D {\boldsymbol u}|^{p-2} D {\boldsymbol u} : D {\boldsymbol \varphi} \frac{|D {\boldsymbol u}|^{p-1}}{\varepsilon+|D {\boldsymbol u}|^{p-1}} \, dx
\\\nonumber& +  \int_{\Omega} |D {\boldsymbol u}|^{p-2} \varepsilon {\boldsymbol \varphi}  \frac{ D {\boldsymbol u} : D (|D {\boldsymbol u}|^{p-1})}{(\varepsilon+|D {\boldsymbol u}|^{p-1})^2}\, dx.
\end{align}

Observe now that,
\begin{equation}\nonumber
\int_{\Omega} |D {\boldsymbol u}|^{p-2} \varepsilon {\boldsymbol \varphi}  \frac{ D {\boldsymbol u} : D (|D {\boldsymbol u}|^{p-1})}{(\varepsilon+|D {\boldsymbol u}|^{p-1})^2} \, dx\leq \int_{\Omega} |{\boldsymbol \varphi} D( |D { \boldsymbol u }|^{p-1})|\, dx<C, 
\end{equation}
due to  $|D { \boldsymbol u }|^{p-1} \in W^{1,2}_{loc}(\Omega)$. Moreover, for $\varepsilon\to 0$, by using the Lebesgue’s Dominated Convergence Theorem in \eqref{eq:lemmaZuzero1}, we get
$$
\int_{\Omega \setminus \{D {\boldsymbol u} =0\} } {\boldsymbol \varphi} \cdot {\boldsymbol f }\, dx
 = \int_{\Omega} |D  {\boldsymbol u}|^{p-2} D {\boldsymbol u } : D {\boldsymbol \varphi} \, dx= \int_{\Omega} {\boldsymbol \varphi} \cdot {\boldsymbol f}\, dx.
$$
This implies that ${\boldsymbol f} (x)=0$ a.e. where $D {\boldsymbol u }(x)=0$. Since ${\boldsymbol  f} $ is a positive function in $\Omega$, we arrived at the desired conclusion.
\end{proof}

In the following result, we are able to show that $\Omega \setminus \mathcal{Z}_u$ is connected.

\begin{proposition}\label{lem:connesso}
Let $\boldsymbol{u}\in C^1(\overline \Omega)$ be a weak solution to the system
\eqref{system}. Assume that {$|D \boldsymbol{u}|^{p-2}D \boldsymbol{u} \in W^{1,2}_{loc}(\Omega) $} and  that ${\boldsymbol f}$ is a positive function in $\Omega$.
Then the set $\Omega\setminus \mathcal{Z}_u$ does not contain any connected component $\mathcal C$ such that 
$\overline {\mathcal C} \subset \Omega.$ Moreover if $\Omega$ is a smooth bounded domain with connected boundary, it follows that $\Omega\setminus \mathcal{Z}_u$ is connected.
\end{proposition}
\begin{proof}
By contradiction, let us assume that a such component $\mathcal C$ exists, that is let $\mathcal C$ such that
\[\mathcal C\subset \Omega \qquad \text{with}\qquad \partial \mathcal C \subset \mathcal{Z}_u.\] 
For all $\varepsilon>0$, let us define  $J_\varepsilon:\mathbb{R}^+\cup\{0\}\rightarrow\mathbb{R}$ by setting
\begin{equation}\label{eq:G}
J_\varepsilon(t)=\begin{cases} t & \text{if  $t\geq 2\varepsilon$} \\
2t-2\varepsilon& \text{if $\varepsilon\leq t\leq2\varepsilon$}
\\ 0 & \text{if $0\leq t\leq \e$}.
\end{cases}
\end{equation}
We shall use
\[\varphi=(0,\ldots,\Psi^\ell,0,\ldots,0) \]
as a test function in \eqref{system}, with $\ell\in\{1,\ldots,N\}$, where
\begin{equation}\label{eq:ghdsagadsjgdj}\Psi^\ell=\displaystyle\frac{J_\varepsilon(|\nabla u^\ell|)}{|\nabla u^\ell|}\chi_{\mathcal C}.\end{equation}
Moreover, since $\supp \Psi^\ell \subset \mathcal C$ we have that $\varphi\in W^{1,p}_0(\mathcal C, \mathbb R^N)$. Integrating by parts we get
\begin{equation}\nonumber
\int_{\mathcal C}|D {\boldsymbol u}|^{p-2} \nabla u^\ell\cdot \nabla \Psi^\ell dx= \int_{\mathcal C}f^\ell(x,u)\Psi^\ell dx
\end{equation}
namely
\begin{equation}\label{eq:4g1}
\int_{\mathcal C}|D {\boldsymbol u}|^{p-2} \nabla u^\ell\cdot \nabla \left(\frac{J_\varepsilon(|\nabla u^\ell|)}{|\nabla u^\ell|} \right) dx= \int_{\mathcal C}f^\ell(x,u)\Psi^\ell dx
\end{equation}
Remarkably, using the test function $\Psi^\ell$ defined in \eqref{eq:ghdsagadsjgdj}, we are able to integrate on the boundary~$\partial \mathcal C$ which  could be not regular.   We estimate the first term on the left-hand side of \eqref{eq:4g1}. Denoting $h_\varepsilon(t)=J_\varepsilon(t)/t$, we have 
\begin{eqnarray}\label{eq:4g2}
&&\left| \int_{\mathcal C}|D {\boldsymbol u}|^{p-2} \nabla u^\ell\cdot \nabla \left(\frac{J_\varepsilon(|\nabla u^\ell|)}{|\nabla u^\ell|} \right) dx\right|\\\nonumber
&&\leq C\int_{\mathcal C}|D {\boldsymbol u}|^{p-2}\Big(|\nabla u^\ell|h_\varepsilon'(|\nabla u^\ell|)\Big)||D^2 u^\ell|| dx.
\end{eqnarray}
We claim
\begin{itemize}
\item [$(i)$] $|D {\boldsymbol u}|^{p-2}||D^2 u||  \in L^1(\mathcal C)$;

\

\item [$(ii)$] $|\nabla u^\ell|h_\varepsilon'(|\nabla u^\ell|)\rightarrow 0$ a.e. in $\mathcal C$ as $\varepsilon \rightarrow 0$ and $|\nabla u^\ell|h_\varepsilon'(|\nabla u^\ell|)\leq C$ with $C$ not depending on $\varepsilon$.
\end{itemize}
Let us  prove $(i)$. Combining H\"older's inequality and Lemma~\ref{lemmaL2loc}, it follows
\begin{equation}\label{eq:smm4}
\int_{\mathcal C}|D {\boldsymbol u}|^{p-2}||D^2u^\ell|| \, dx\leq C(\mathcal C)\left( \int_{\mathcal C}|D {\boldsymbol u}|^{2(p-2)}||D^2u^\ell||^2\right)^{\frac 12} \, dx \leq C
\end{equation}
\

Let us  prove $(ii)$. Exploiting the definition \eqref{eq:G}, by straightforward calculation we obtain
$$
h'_\varepsilon(t)=
\begin{cases} 0 & \text{if  $t> 2\varepsilon$} \\
\frac{2\varepsilon}{t^2}& \text{if $\varepsilon< t<2\varepsilon$}
\\ 0 & \text{if $0\leq t< \e$},
\end{cases}
$$
and then we have $|\nabla u^\ell|h_\varepsilon'(|\nabla u^\ell|)\rightarrow 0$ a.e. for $\varepsilon\rightarrow 0$ in $\mathcal C$ and  $|\nabla u^\ell|h_\varepsilon'(|\nabla u^\ell|)\leq 2$.

\

Then, using dominated convergence and  equation \eqref{eq:4g2} we have
\[\int_{\mathcal C}|D {\boldsymbol u}|^{p-2} \nabla u^\ell\cdot \nabla \left(\frac{J_\varepsilon(|\nabla u^\ell|)}{|\nabla u^\ell|} \right) dx\rightarrow 0,\]
as $\varepsilon \rightarrow 0$ and therefore from \eqref{eq:4g1} we get a contradiction since for $\varepsilon \rightarrow 0$ we have that 
\[\int_{\mathcal C}f^\ell(x,u)\, dx>0.
\]

If $\Omega$ is smooth, we can apply Theorem~\ref{thm:hopf} to state that a neighborhood of the boundary belongs to a component $\mathcal{C}$ of $\Omega \setminus \mathcal{Z}_u$. By what we have proved previously, there not exist a second component $\tilde{\mathcal{C}}$. Therefore, $\Omega \setminus \mathcal{Z}_u$ is connected as desired.

\end{proof}

\

Finally, we have collected all the tools to apply the moving-planes technique and complete the proof of both main theorems.

\begin{proof}[Proof of Theorem~\ref{thm:main1}]
Since $\boldsymbol{u}>0$ in $\Omega$ and $\boldsymbol{u}=0$ on $\partial \Omega$, then $\boldsymbol{u}\leq \boldsymbol{u}_\lambda$ on $\partial \Omega_\lambda$ for $\lambda\in(a,0)$, where  $\Omega_\lambda$  is defined in \eqref{eq:whawha}. We also have that $\boldsymbol{u}_\lambda$ satisfies the equation 
$$
-\boldsymbol{\Delta}_p \boldsymbol{u}_\lambda = f(x_\lambda,\boldsymbol{u}_\lambda) \quad\mbox{in } \Omega_\lambda.
$$
Using Hopf's Lemma, see Theorem \ref{thm:hopf} we have that $u^\ell$, for all $\ell=1,\ldots,N$, are strictly increasing in the $x_n$-direction, near the boundary $\partial \Omega$. Namely, there exists  some $\varepsilon>0$ such that for $a<\lambda<a+\varepsilon$, it holds  
\[\boldsymbol{u}\leq \boldsymbol{u}_\lambda \quad \text{in }\Omega_\lambda, \] 
where we have defined $a$ in \eqref{eq:a}.

As a result, the set 
\[\Lambda:=\Big\{\lambda \in \mathbb R  \,: \, \boldsymbol{u}\leq \boldsymbol{u}_t \,\, \text{in}\,\, \Omega_\lambda\,:\, a<t\leq\lambda\Big\}\]
is not empty. Let us now set
\[\bar \lambda= \sup \Lambda.\] 

Next, we shall prove the following 

\

\noindent {\sc Claim:} $\bar\lambda =0$. Assume the contrary, namely let us suppose that $\bar\lambda <0$. By continuity we have that 
\[\boldsymbol{u}\leq \boldsymbol{u}_{\bar{\lambda}} \quad\text{in } \Omega_{\bar\lambda}.\]
By using  \eqref{eq:fcresc}, we have that in the weak sense 
\[-\boldsymbol{\Delta}_p \boldsymbol{u}_{\bar\lambda} =f(x_{\bar\lambda},\boldsymbol{u}_{\bar\lambda})>f(x,\boldsymbol{u}_{\bar\lambda})\quad \text{in } \Omega_{\bar\lambda}.\] Consequently, since
\[-\boldsymbol{\Delta}_p \boldsymbol{u} -f(x,\boldsymbol{u})<-\boldsymbol{\Delta}_p \boldsymbol{u}_{\bar\lambda}-f(x,\boldsymbol{u}_{\bar\lambda}) \quad \text{in } \Omega_{\bar\lambda}.\]
by strong comparison principle, see Theorem~\ref{thm:strong}, we deduce 
\[\boldsymbol{u}< \boldsymbol{u}_{\bar{\lambda}} \quad\text{in } \Omega_{\bar\lambda}\setminus (\mathcal{Z}_{u} \cup \mathcal{Z}_{u_{\bar\lambda}} ). \]
Next, fix a compact set $K$ so large such that 
\[\boldsymbol{u}<\boldsymbol{u}_{\bar \lambda +\varepsilon}\,\, \text{in}\,\, K\Subset \left(\Omega_{\bar \lambda} \setminus (\mathcal{Z}_{u} \cup \mathcal{Z}_{u_{\bar\lambda}} ) \right).\]
If the Lebeasgue measure  $|K|$ is sufficiently big and $\varepsilon$ small we infer that
\[\Omega_{\bar \lambda +\varepsilon}\setminus K := \tilde \Omega,\]
is such that 
 $\tilde \Omega\subseteq\Omega$, \[\bigcup_{x'\in \Omega'} Z_{x'}^{\varepsilon} \cup L_{x'}^\theta\subseteq\tilde\Omega\]
where, eventually taking $K$ bigger and $\varepsilon $ smaller, the sets  $Z_{x'},L_{x'}$ satisfy  assumptions $(i)$ and $(ii)$ of the weak comparison principle, Theorem \ref{thm:weak1}. Moreover, let us observe that $\boldsymbol{u}\leq \boldsymbol{u}_{\bar{\lambda}+\varepsilon} $ in $\partial \tilde\Omega$. 
Then, by Theorem \ref{thm:weak1}, we deduce that 
\[\boldsymbol{u}\leq \boldsymbol{u}_{\bar\lambda+\varepsilon}\,\, \text{in}\,\, \Omega_{\bar \lambda+\varepsilon}\setminus K\]
and hence $\boldsymbol{u}\leq \boldsymbol{u}_{\bar\lambda+\varepsilon}$ in $\Omega_{\bar \lambda+\varepsilon}$, for all $0<\varepsilon\leq\bar \varepsilon$, with $\bar\varepsilon\in \mathbb R$ small.
 This fact contradicts  the assumption $\bar\lambda=\sup \Lambda$. Thus
\[\bar \lambda=0.\]
Repeating the proof in the opposite direction, we obtain that 
\[ \boldsymbol{u} \equiv \boldsymbol{u}_{0},\]
namely we get the symmetry of solutions, with respect to the hyperplane $T_0$, in the $x_n$-direction. 
Now, let us prove that the solution is  monotone increasing with respect to  the $x_n$-direction in $\Omega_0=\{x\in \Omega \, :\, x_n<0\}$. Let $(x',s_1), (x',s_2) \in \Omega_0$. Then for $\lambda=(s_1+s_2)/2$ it holds $u \leq u_{\lambda}$ in $\Omega_{\lambda}$, so 
\begin{equation*}
u(x',s_1)\leq u(x',s_2).
\end{equation*}
\end{proof}
\begin{proof}[Proof of Theorem~\ref{thm:main2}]
The proof follows straightforward from the given in the previous Theorem. It suffices to stress that by \eqref{eq:monface}, one obtains that $\boldsymbol{u}\leq \boldsymbol{u}_\lambda$ on $\partial \mathcal{P}_\lambda$ for $\lambda\in(a,0)$, including the faces $\mathcal{P}^\bot_i$ for $i=-1,1$. This allows us to apply the Hopf's Lemma, weak comparison principle given in Theorem~\ref{thm:weak1} and Theorem~\ref{thm:strong}, jointly to \eqref{eq:monface}, in the corresponding steps.
\end{proof}

One can obtain the same conclusion of Theorem~\ref{thm:main2} introducing an integral condition on the edges instead of the monotonicity of the solution on $\mathcal{P}^\bot_i$. Actually, this condition allows us to obtain a weak comparison principle and, then, start the moving-planes procedure.
\begin{theorem}
Let $\mathcal{P}$ be a bounded pipe type domain of $\mathbb{R}^n$, $n\geq 2$, satisfying \ref{hpP}. Let $\boldsymbol{u}$ be a $C^1(\overline\Omega)$ weak solution to \eqref{system2} with $1<p\leq 2$ such that $|D\boldsymbol{ u }|^{p-2} D\boldsymbol{u} \in W^{1,2} (\Omega)$, $\boldsymbol{f}$ verifies \eqref{hpf}, \eqref{eq:fcresc} and \eqref{eq:fsym}. Suppose that for every $\lambda\in (a,0]$
\begin{equation}\label{eq:hipfront}
\int_{\mathcal{P}^\bot_{-1} \cup \mathcal{P}^\bot_{1}} |D {\boldsymbol u}|^{p-2}\frac{\partial \boldsymbol{u}}{\partial \hat n}:(\boldsymbol{u}-\boldsymbol{u}_\lambda)^+d\sigma=0,
\end{equation}
where $\hat n$ denotes the outer unit vector to the subset $\mathcal{P}^\bot_{i}$ with $i=-1,1$.
\end{theorem}
\begin{proof}
First, let us prove that $(\boldsymbol{u}-\boldsymbol{u}_\lambda)^+$ is a suitable test function for \eqref{system2}. Due to the fact that $|D\boldsymbol{ u }|^{p-2} D\boldsymbol{u} \in W^{1,2} (\Omega)$, we have that $-{\boldsymbol \Delta}_p{\boldsymbol u}=   {\boldsymbol f}(x,{\boldsymbol u})$ a. e. in $\mathcal{P}$. Then,
by using the divergence theorem, we have that
\begin{align}\label{eq:hipfront1}
&\int_{ \mathcal{P} } |D\boldsymbol{u}|^{p-2} D\boldsymbol{u} :D (\boldsymbol{u}-\boldsymbol{u}_\lambda)^+dx-  \int_{\partial \mathcal{P} } |D \boldsymbol {u}|^{p-2}\frac{\partial \boldsymbol{u}}{\partial \hat n}:(\boldsymbol{u}-\boldsymbol{u}_\lambda)^+d\sigma= \nonumber\\
& -\int_{ \mathcal{P} } \boldsymbol{div}(|D\boldsymbol{u}|^{p-2} D\boldsymbol{u} )\cdot(\boldsymbol{u}-\boldsymbol{u}_\lambda)^+dx= \int_{ \mathcal{P} } {\boldsymbol f}(x,{\boldsymbol u}) \cdot (\boldsymbol{u}-\boldsymbol{u}_\lambda)^+\, dx .
\end{align}

Since $u_\lambda>0$ in $\mathcal{P}$ and $u=0$ on $\partial_w \mathcal{P}$, then $(\boldsymbol{u}-\boldsymbol{u}_\lambda)^+\equiv 0$ on $\partial_w \mathcal{P}$. This fact, with \eqref{eq:hipfront} gives that the second term in the left hand side of \eqref{eq:hipfront1} vanishes. Then we obtain
\begin{equation}\nonumber
 \int_{ \mathcal{P} } |D\boldsymbol{u}|^{p-2} D\boldsymbol{u} :D (\boldsymbol{u}-\boldsymbol{u}_\lambda)^+dx = \int_{ \mathcal{P} } \boldsymbol{f}(x,\boldsymbol{u})\cdot(\boldsymbol{u}-\boldsymbol{u}_\lambda)^+dx.
\end{equation}
Therefore, $(\boldsymbol{u}-\boldsymbol{u}_\lambda)^+$ is a proper test function for \eqref{system2}. Following the proof of weak comparison principle, Theorem~\ref{weak}, we are in position to establish the same thesis. Moreover, the proof follows the one of Theorem \ref{thm:main2}.
\end{proof}

\vspace{1cm}

\begin{center}{\bf Acknowledgements}\end{center} This paper is part of the fellowship `José Castillejo' CAS19/00356 supported by Ministerio de Ciencia y Universidades (Spain). R. López-Soriano spent a period as visiting researcher at Università della Calabria, he is grateful to the institution for the kind hospitality and thanks L. Montoro and B. Sciunzi for the camaraderie. \\ R. López-Soriano is partially supported by  Agencia Estatal de Investigación (Spain), project PID2019-106122GB-I00/AEI/10.3039/501100011033. L. Montoro and B. Sciunzi are partially supported by PRIN project 2017JPCAPN (Italy): Qualitative and quantitative aspects of nonlinear PDEs, and L. Montoro by  Agencia Estatal de Investigación (Spain), project PDI2019-110712GB-100.

\


\begin{thebibliography}{99}

\bibitem{BeCr1}{H. Beirão da Veiga, F. Crispo}. {On the global $W^{2,q}$ regularity for nonlinear $N$-systems of the $p$-Laplacian type in n space variables, }{\emph{Nonlinear Anal.} 75
(2012) 4346--4354}.

\bibitem{BeCr2}{H. Beirão da Veiga, F. Crispo}. {On the global regularity for nonlinear systems of the p-Laplacian type, }{ \emph{Discrete Contin. Dyn. Syst. Ser.} S 6 (2013) 1173--1191.}



\bibitem{BerDiRu}{L.C. Berselli, L. Diening, M. Růžička}. {Existence of strong solutions for incompressible fluids with shear dependent viscosities, }{\emph{J. Math. Fluid Mech.} 12
(2010) 101--132.}


\bibitem{BiAmHas}{R. B. Bird, R. C. Armstrong, O. Hassager}. {\emph{Dynamic of Polymer Liquids, 2nd edition}, John Wiley, 1987.}

\bibitem{Cma}{A. Cianchi, V. Maz'ya}. {Optimal second-order regularity for the p-Laplace system, }{\emph{J. Math. Pures Appl.} (9) 132 (2019), 41–78}.

\bibitem{Cri}{F. Crispo, P. Maremonti}. {A high regularity result of solutions to modified -Stokes equations, }{Nonlinear Anal-Theor, vol. 118 (2015), 97--129}.

\bibitem{Dam}{L. Damascelli}. {Comparison theorems for some quasilinear degenerate elliptic operators and applications to symmetry and monotonicity results, }{\emph{Ann. Inst. H. Poincaré. Anal. Non Linéaire}, 15 (1998), no. 4, 493--516.}

\bibitem{DamPa}{L. Damascelli, F. Pacella}. {Monotonicity and symmetry of solutions of p-Laplace equations $1<p<2$ via the moving plane method, }{\emph{Ann. Scuola Norm. Sup. Pisa Cl. Sci.} 26 (4) (1998) 689--707.}


\bibitem{DamSci}{L. Damascelli, B. Sciunzi}. {Regularity, monotonicity and symmetry of positive solutions of $m$-Laplace equations, }{\emph{ J. Differential Equations} 206 (2004), no. 2, 483--515.}

%



\bibitem{GT} D.~Gilbarg, N.~S.~Trudinger. \emph{Elliptic Partial Differential Equations of Second Order}.
\newblock {Springer, Reprint of the 1998 Edition}.

\bibitem{Lions1}{J.-L. Lions}. {Sur certaines équations paraboliques non linéaires}, {\emph{Bull. Soc. Math. France} 93 (1965) 155--175.}


\bibitem{Lions2}{J.-L. Lions}. {\emph{Quelques Méthodes de Résolution des Problèmes aux Limites non Linéaires}}, {Dunod, Gauthier-Villars, Paris, 1969.}




\bibitem{KuuMin}{T. Kuusi, G. Mingione}. {A nonlinear Stein theorem, }{\emph{Calc. Var. Partial Differential Equations} 51 (2014), no. 1-2, 45--86.}

\bibitem{KuuMin2}{T. Kuusi, G. Mingione}. {Vectorial nonlinear potential theory, }{\emph{J. Eur. Math. Soc.} (JEMS) 20 (2018), no. 4, 929--1004.}

\bibitem{Mal1}{J. Málek, K. R. Rajagopal}. {Mathematical issues concerning the Navier–Stokes equations and some of its generalizations, in: Evolutionary equations, }{Vol. II, 371–459, Handb. Differ. Equ., Elsevier/North-Holland, Amsterdam, 2005.}

\bibitem{Mal2}{J. Málek, K. R. Rajagopal, M. Růžička}. {Existence and regularity of solutions and the stability of the rest state for fluids with shear dependent viscosity, }{\emph{Math. Models Methods
Appl. Sci.} 5 (1995), 789–812.}

\bibitem{Minsur}{G. Mingione}. {Regularity of minima: an invitation to the dark side of the calculus of variations, }{\emph{Appl. Math.} 51 (2006), no. 4, 355–426.}

 \bibitem{PSB} {P.~Pucci and J.~Serrin}.
\newblock \emph{The maximum principle}.
\newblock Birkhauser, Boston, 2007.



 \bibitem{Troia} {G.M.~Troianiello}.
\newblock \emph{Elliptic Differential Equations and Obstacle Problems}.
\newblock Plenum Press, New York, 1987.

\end{thebibliography}
\end{document}